\documentclass[11pt]{amsart}
\usepackage{amsmath,amssymb,amsthm,hyperref}
\usepackage{graphicx}
\usepackage{tikz}  
\usepackage{comment}
 \newtheorem{theorem}{Theorem}[section]
 \newtheorem{corollary}[theorem]{Corollary}
 
 \newtheorem{lemma}[theorem]{Lemma}
 \newtheorem{proposition}[theorem]{Proposition}
 
 {\theoremstyle{definition}
 \newtheorem{definition}[theorem]{Definition}
 \newtheorem{remark}[theorem]{Remark}

 }

\newcommand{\black}{\color{black}}

\def\R{\mathbb{R}}

\def\EE{\mathbb{E}}\def\E{\mathbb{E}}

\renewcommand{\qed}{\hfill$\square$}

\title[Riemann approximation of fractional integrals]{\black \small Asymptotic error distribution for the Riemann approximation of integrals driven by fractional Brownian motion}

\author{Valentin Garino}
\address{Valentin Garino, Universit\'e du Luxembourg, 
Unit\'e de Recherche en Math\'ematiques,
Maison du Nombre,
6 avenue de la Fonte,
L-4364 Esch-sur-Alzette,
Grand Duch\'e du Luxembourg}
\email{valentin.garino@uni.lu} 
\thanks{}
\author{Ivan Nourdin}
\address{Ivan Nourdin, Universit\'e du Luxembourg, 
Unit\'e de Recherche en Math\'ematiques,
Maison du Nombre,
6 avenue de la Fonte,
L-4364 Esch-sur-Alzette,
Grand Duch\'e du Luxembourg}
\email{ivan.nourdin@uni.lu} 
\thanks{}
\author{Pierre Vallois}
\address{Pierre Vallois, Universit\'e de Lorraine
Institut Elie Cartan de Lorraine, 
UMR 7502
F-54506 Vandoeuvre-l\`es-Nancy, 
France}
\email{pierre.vallois@univ-lorraine.fr} 
\thanks{}

\date{\today}
\begin{document}
\maketitle

\medskip

\begin{abstract}

We consider Riemann sum approximations of stochastic integrals with respect to the fractional Browian motion of index $H\geq \frac12$. 
We show the convergence of these schemes at first and second order. 
The processes obtained in the limit in the second case are stochastic integrals with respect to the Rosenblatt process if $H >\frac34$ and the standard Brownian motion otherwise. 
These results are obtained under the assumption that the integrand is a ``controlled'' process. 
We provide many examples of such processes, in particular fractional semimartingales and multiple Wiener-It\^o integrals.\\

\end{abstract}

\section{Introduction}

Fractional Brownian motion was introduced by Kolmogorov \cite{kolmogorov} in the 40's. 
Mandelbrot and Van Ness \cite{mandelbrot} popularized it and gave some quantitative properties.
Since then, its range of applications has been steadily growing: for example, nowadays it can serve to recreate certain natural landscapes (such as submarine floors, see \cite{PP98}) or to model rainfalls (see \cite{VBN96}). It also often serves as a model in hydrology (e.g. \cite{molz}), telecommunications (e.g. \cite{leland,mikosch}), finance (e.g. \cite{comte}) or physics (e.g. \cite{vojta}), to name but a few.
Since the explicit calculation of stochastic integrals driven by fractional Brownian motion is impossible except in very particular cases, 
 it is natural to try to approximate these integrals by Riemann sums and to study their convergence.

\medskip

In \cite{rootzen1980limit}, Rootz\'en considered the It\^o integral $\int_0^t u_sdB_s$ of an adapted integrand $u$ with respect to a standard Brownian motion $B$, and investigated the asymptotic behavior of the
approximation error $\int_0^t u_s dB_s - \int_0^t u^n_s dB$ when $u^n$ are approximating integrands (for instance, we can choose $u^n$ so that $\int_0^t u^n_s dB_s$ corresponds to the Riemann sum associated with  $\int_0^t u_s dB_s$). Using It\^o stochastic calculus, Rootz\'en \cite{rootzen1980limit} exhibits after proper normalisation a {\it stable} limit of the form $\int_0^ta_sdW_s$, with $W$ a Brownian motion {\it independent} of $B$.
As an illustration, he applied his abstract result to prove a functional central limit-type theorem in the space $\mathcal{D}_{\mathbb{R}}([0,T])$ of c\`adl\`ag functions equipped with the Skorohod topology, and with $u_s=f(B_s)$ (provided $f$ is smooth and bounded enough):
 \begin{eqnarray}
\notag
&&\sqrt{n}\left(\int_0^t f(B_s)dB_s-\!\!\!\sum_{k=0}^{\lfloor nt \rfloor-1}f(B_{\frac{k}{n}})(B_{\frac{k+1}{n}}-B_{\frac{k}{n}})\right)_{t\in[0,T]} \\
&\overset{\rm stably}{\underset{n\rightarrow\infty}{\longrightarrow}}& \left(\sqrt{\frac{1}{2}}\int_0^t f'(B_s)dW_s\right)_{t\in[0,T]}.\label{illustrationintro}
\end{eqnarray}

Rootz\'en's work  \cite{rootzen1980limit} paved the way for a new area of research on the subject and related topics. For example, we can mention multidimensional extensions (see \cite{lindberg2013error}), generalizations to the case of random discretisation times (see \cite{fukasawa2011discretization}), applications in finance (see \cite{gobet2001discrete}) and approximation schemes of stochastic differential equations (SDEs) driven by semimartingales (see \cite{jacod1998asymptotic}). The recent paper \cite{alos2018asymptotic}  provides an asymptotic expansion for the weak discretization error of It\^o's integrals.

\smallskip

Approximation schemes for SDEs driven by a fractional Brownian motion has been addressed in  \cite{hu2016rate,neuenkirch2007exact}. 
But Riemann sums approximations of stochastic integrals with respect to fractional Browian motion, as done by Rootz\'en \cite{rootzen1980limit} in the case of  the standard Brownian motion, had not yet been studied; 
the aim of this article is to fill this gap.

\medskip

In the present paper, we deal with a fractional Brownian motion $B$  of Hurst index $H\in\big[\frac12,1\big)$. 
All the processes considered in this paper will always be implicitly assumed to be measurable with respect to $B$.
Also, note that the range of $H$ includes $\frac12$ (corresponding to Brownian motion), which will allow us to compare our results with those of  \cite{rootzen1980limit}.
Our goal is to analyze the fluctuations around the approximation by Riemann sums of stochastic integrals with respect to a fractional Brownian motion. 
We will set up an approach based on two main steps.
\begin{itemize}
\item \textit{\underline{Step 1}: weighted limit theorem}. Let $(u^n)$ be a sequence of processes of the form $u^n=\sum_{k=1}^{\lfloor n\cdot\rfloor}X^n_k$ for which a functional convergence 
$
u^n\rightarrow w
$
holds.
We extend  this convergence to
$$\sum_{k=1}^{\lfloor n\cdot\rfloor}h_\frac{k}{n}X^n_k\longrightarrow\int_0^\cdot h_sdw_s$$ 
for a given class of appropriate random processes $h$, and where the nature of the integral with respect to $w$ (It\^o, Young, etc.) is chosen according to the features of $w$. 
When the sequence $(X^n_k)$ is built from the increments of a fractional Brownian motion, this type of questions has received some important contributions in recent years, see e.g. \cite{liu-tindel} and the references therein. We also mention \cite{hu2016rate}, which was actually our main inspiration for this step.

\item \textit{\underline{Step 2}: Taylor expansion}. To perform Step 1, we assume that our integrand $u$ is `controlled' by the increments of the integrator $B$, in the sense that there is a process $h$ and a remainder $r$ such that $u_t=u_s+h_s(B_t-B_s)+r_{s,t}$ for any $t\geq s$. 
These types of Taylor-like expansions are strongly related with the notion of controlled paths studied in the rough path theory, see \cite{gubinelli2004controlling}.
We will characterize precisely the set of such processes below.
\end{itemize}

\medskip

The statement of the two main Theorems \ref{main1} and \ref{main2} require the introduction of notations:

{\bf (i)} a $d$-dimensional fractional Brownian motion $B$ (for some $d\in\mathbb{N}^*$) of Hurst index $H\in\big[\frac12,1\big)$ (as already mentioned, 
all the processes considered in this paper are implicitly assumed to be measurable with respect to $B$);

{\bf (ii)} an $m$-dimensional process $u$, with the property that the stochastic integrals $\int_0^t u^i_sdB^j_s$, $1\leq i\leq m$, $1\leq j\leq d$, are well-defined. 
At this stage, we note that the integrals $\int u^idB^j$ must be understood in the Young sense  when $H>\frac12$ and in the It\^o sense when $H=\frac12$. Precise statements will be given later on.

{\bf (iii)} our quantity of interest: for $t\in [0,T],\,1\leq i\leq m, \,1\leq j\leq d$,
\begin{equation}\label{mnij}
M^{n,i,j}_t=n^{2H-1}\left(\int_0^tu^i_sdB^j_s-\sum_{k=0}^{nt_n}u^i_\frac{k}{n}\left(B^j_{\frac{k+1}{n}\wedge t}-B^j_\frac{k}{n}\right)\right).
\end{equation}
In (\ref{mnij}) and in all what follows,  we write $t_n=\frac{\lfloor nt\rfloor}{n}$ when $t\in\R_+$ and $n\in\mathbb{N}\setminus\{0\}$.

{\bf (iv)} the correlation function: for all $t\geq s$ and all  $y\geq x$,
\begin{eqnarray*}
r_H(s,t,x,y)&=&\mathbb{E}\left[(B^1_t-B^1_s)(B^1_y-B^1_x)\right]\\
&=&\frac{1}{2}\left(|t-x|^{2H}+|s-y|^{2H}-|s-x|^{2H}-|t-y|^{2H}\right);
\end{eqnarray*}

{\bf (v)} the rate function at zero
\begin{equation*}\kappa_H(v):=\left\{
      \begin{array}{cll}
        \sqrt{v} &\textit{ if } H\in[\frac12,\frac34)\\
        \sqrt{v\ln{\frac{1}{v}}}& \textit{ if } H=\frac{3}{4}\\
        v^{2-2H} &\textit{ if } H\in(\frac34,1)\\
      \end{array}
    \right.,\quad v\in(0,1];\label{kappa}
    \end{equation*}

{\bf (vi)} the rate function at infinity  
\begin{equation*}\nu_H(n):=\left\{
      \begin{array}{cll}
        \sqrt{n} &\textit{ if } H\in[\frac12,\frac34)\\
        \sqrt{n/\ln{n}} & \textit{ if } H=\frac{3}{4}\\
        n^{2-2H}  &\textit{ if } H\in(\frac34,1)\\
      \end{array}
    \right.,\quad n\geq 1.
    \label{nu}
    \end{equation*}

In addition, we  assume that the process $u$ considered in point {\bf (ii)} satisfies a {\it structural condition}, 
that we describe now.
Set 
\begin{eqnarray*}
f_1(s,t,x,y)&=&|t-s|^{2H-1}|x-y|^{2H-1}r_H(s,t,x,y);\\
 f_2(s,t,x,y)&=&f_1(s,t,x,y)\kappa_H(|t-s|)\kappa_H(|x-y|).
\end{eqnarray*}
We introduce the two following spaces $\mathbb{C}_1$ and $\mathbb{C}_2$ of \textit{pseudo-controlled paths}.
\begin{definition}\label{defpseudo}
{\it
Fix $a\in\{1,2\}$.
We say that the pair $(u,P)$ belongs to $\mathbb{C}_a$  if:
\begin{itemize}
\item $P=(P^{i,j}_t)_{t\in[0,T],1\leq i\leq m,1\leq j\leq d}$ is an $(m\times d)$-dimensional  process ;
\item $\int_s^t u^i_rdB^j_r$ is well-defined for any $1\leq i\leq m$ and $1\leq j\leq d$; 
\item \begin{equation}\label{pseudo}\mathbb{E}\left[L_{s,t}^{i,j}\,L_{x,y}^{i,j}\right]=o(f_a(s,t,x,y))\end{equation} for all $1\leq i\leq m$ and $1\leq j\leq d$, uniformly on $(s,t,x,y)\in[0,T]^4$ such that $s\leq t$ and $x\leq y$ as $|t-s|+|x-y|\rightarrow 0$,
where
\begin{equation}
L_{s,t}^{i,j}=\int_s^t \left\{u^i_r-u^i_s-\sum_{k=1}^dP^{i,k}_s(B^k_r-B^k_s)\right\}dB^j_r.
\label{Lij}
\end{equation}
\end{itemize}

}
\end{definition}
We note the obvious inclusion $\mathbb{C}_2\subset\mathbb{C}_1$.
We give two examples to understand Definition \ref{defpseudo}.
For the first one, we consider the case where each component $u^i$ of $u$ is a ``fractional semimartingale'', namely
\begin{equation*}\label{class1}
u^i_t=u^i_0 + \sum_{j=1}^d\int_0^t a^{i,j}_s dB^j_s + \int_0^t b^i_s ds, \quad t\in [0,T].
\end{equation*}
Then, under certain assumptions on $a$ and $b$ (see Section \ref{sec:Controlled paths} for precise statements), the pair $(u,a)$ belongs to $\mathbb{C}_2$ with $a=P$.

For the second one, we assume that $m=d=1$ (for simplicity) and that $u$ has the form of a multiple Wiener-It\^o integral of order $q\geq 1$; then, with $P_s=D_su_s$ (where $D$ indicates the Malliavin derivative) 
and under some conditions, the pair $(u,P)$ belongs to $\mathbb{C}_2$, see Section \ref{MWI} for precise statements.

\medskip

We can now state our two main results.
The framework of Theorem \ref{main1} is general (assuming that the pair $(u, P)$ belongs to $\mathbb{C}_1$ and satisfies other technical conditions)
and concerns the convergence of $M^{n,i,j}$ as $n\to\infty$ in {\it probability}, towards an identified limit.
The situation where $H>\frac12$ differs significantly from $H = \frac12 $, because in this latter case $M^{n,i,j}$  converges {\it in law} (but not in probability, because of the creation of an independent alea, see e.g. (\ref{illustrationintro})).

\begin{theorem}\label{main1} (First order convergence)
Fix $H\in(\frac{1}{2},1)$ and let $(u,P)\in\mathbb{C}_1$ be such that  $P$ is a.s.\! continuous and satisfies $\mathbb{E}\left[\|P\|^{2+\gamma}_{\infty}\right]<+\infty$ for some $\gamma>0$. (Here and throughout the paper, we write $\|\cdot \|_{\infty}$ to indicate the uniform norm over $[0,T]$.)
Then, uniformly on $[0,T]$ in probability,
\begin{equation}
\left\{M_\cdot^{n,i,j}\right\}_{1\leq i\leq m,1\leq j\leq d}\underset{n\to\infty}{\longrightarrow}\left\{\frac{1}{2}\int_0^\cdot P^{i,j}_sds\right\}_{1\leq i\leq m,1\leq j\leq d}.
\label{ordre1}
\end{equation}
Moreover, this convergence also holds in $L^2(\Omega)$ for any fixed $t\in [0,T]$.
\end{theorem}

Theorem \ref{main1} give sufficient conditions for (\ref{ordre1}) to take place.  
These conditions are however not necessary: we 
develop in Section \ref{sec3.4} an example where the assumptions of Theorem \ref{main1} are not satisfied whereas the convergence 
(\ref{ordre1}) holds.

Let us now study the fluctuations of $M_\cdot^{n,i,j}$ around its limit.

\begin{theorem}[Second order convergence]\label{main2}
Fix $H\in[\frac12,1)$, and let $Z=(Z^{k,j})_{1\leq k,j\leq d}$ (resp. $W=(W^{k,j})_{1\leq k,j\leq d}$) denote the matrix-valued Rosenblatt process measurable with respect to $B$ (resp. the matrix-valued Brownian motion independent from $B$) constructed in Section \ref{rosenblatt-and-brownian}.
\smallbreak (A)  {\rm [non-Brownian case $H>\frac12$]} Assume $(u,P)\in \mathbb{C}_2$, $u$ is $\alpha$-H\"older continuous for some $\alpha>1-H$ and $P$ is $\beta$-H\"older continuous over $[0,T]$ for some $\beta>\frac{1}{2}$.
\begin{itemize}
\item  If $\frac{1}{2}<H\leq\frac{3}{4}$ then, stably in $\mathcal{C}_{\mathbb{R}^{m\times d}}([0,T])$,
\begin{eqnarray*}
&&\left\{\nu_H(n)\left(M_{\cdot}^{n,i,j}-\frac{1}{2}\int_0^\cdot P^{i,j}_sds\right)\right\}_{1\leq i\leq m,1\leq j\leq d}\\
&&\hskip3cm\underset{n\to\infty}{\longrightarrow} \left\{\sum_{k=1}^{d}\int_0^{\cdot}P_s^{i,k}dW_s^{k,j}\right\}_{1\leq i\leq m,1\leq j\leq d},
\end{eqnarray*}
where the integrals in the right-hand side are understood as Wiener integrals.

\item   If $\frac{3}{4}<H<1$, assume in addition that
$\sum_{j=1}^d\sum_{i=1}^m\E\|P^{i,j}\|_{\beta}^{2+\gamma}<\infty$  for some $\gamma>0$
where, here and throughout the paper, $\|\cdot\|_\beta$ indicates the usual $\beta$-H\"older seminorm (see also (\ref{thetaholder})).
 Then, uniformly on $[0,T]$ in probability,
\begin{eqnarray*}
&&\left\{\nu_H(n)\left(M_{\cdot}^{n,i,j}-\frac{1}{2}\int_0^\cdot P^{i,j}_sds\right)\right\}_{1\leq i\leq m,1\leq j\leq d}\\
&&\hskip3cm \underset{n\to\infty}{\longrightarrow}  \left\{\sum_{k=1}^{d}\int_0^{\cdot}P_s^{i,k}dZ_s^{k,j}\right\}_{1\leq i\leq m,1\leq j\leq d},
\end{eqnarray*}
where the integrals in the right-hand side are understood as Young integrals. Moreover, this convergence also holds in $L^2(\Omega)$ for any fixed $t\in [0,T]$.
\end{itemize}
\smallbreak (B) {\rm [Brownian case $H=\frac12$]} Assume  that $(u,P)\in\mathbb{C}_2$, that $u$ and $P$ are progressively measurable, and that $P$ is a.s.\! piecewise continuous with $\mathbb{E}\left[\|P\|^{2+\gamma}_{\infty}\right]<+\infty$ for some $\gamma>0$. Then, stably in $\mathcal{C}_{\mathbb{R}^{m\times d}}([0,T])$,
\begin{equation*}\left\{\nu_H(n)M_{\cdot}^{n,i,j}\right\}_{1\leq i\leq m,1\leq j\leq d}\underset{n\to\infty}{\longrightarrow} \left\{\sum_{k=1}^{d}\int_0^{\cdot}P_s^{i,k}dW_{s}^{k,j}\right\}_{1\leq i\leq m,1\leq j\leq d},
\label{Brownian}
\end{equation*}
where the integrals in the right-hand side are understood as Wiener integrals.
\end{theorem}

In Theorem \ref{main2},
we could have considered non-uniform or even random subdivisions (like done in \cite{fukasawa2011discretization} in the semimartingale context) but this would have led to significant technical complications due to the non-stationarity of the resulting sequence of increments. 
Similarly, we could also have replace the fractional Brownian motion by a
general Gaussian processes with a covariance function assumed to behave locally as that of the fractional Brownian motion.\\

The rest of the paper is organized as follows. 
Section 2 contains some reminders and useful results about Malliavin calculus and fractional integration. 
In Section 3, we discuss in details some examples. 
Finally, the proofs of the main results are given in Section 4.

\bigskip

\section{Preliminaries}

\subsection{Notation}
In the sequel, $\mathbb{N}$ (resp $\mathbb{N}^*$) will denote the space of nonegative (resp strictly positive) integers, $\mathcal{C}^k([0,T])$ (resp $\mathcal{C}_b^k([0,T])$) the space of $k$-times continuously differentiable functions (resp $k$-times continuously differentiable with bounded derivatives) over $[0,T]$, and $\mathcal{C}^{\theta}([0,T])$ the space of $\theta$-H\"older continuous functions (with $\theta\in (0,1)$) endowed with the $\theta$-H\"older seminorm,  i.e
\begin{equation} \|f\|_{\theta}=\sup_{0\leq s<t\leq T}\frac{|f(t)-f(s)|}{|t-s|^\theta}.
\label{thetaholder}
\end{equation}
We also consider the space
$\mathcal{C}_{\mathbb{R}^p}([0,T])$ of functions $[0,T]\rightarrow\mathbb{R}^p$ endowed with the norm $\|\cdot\|_{\infty}$ of uniform convergence over $[0,T]$, the space $\mathcal{D}_{\mathbb{R}^p}([0,T])$ 
of c\`adl\`ag functions endowed with the Skorokod topology $J_1$ and, for $p>0$, the space $L^p(\Omega)$ of random variables endowed with the $L^p(\Omega)$-norm $\|\cdot\|_p$.

\subsection{Reminders of Malliavin calculus}
This section is a condensed summary of some notions presented in  \cite{nourdin2012normal,nualart2006malliavin,pipiras2001classes}.
It is the occasion to fix the notation  used in the paper.
For more details or missing proofs, we refer the reader to the aforementioned references.

Starting from now, we fix once for all an horizon time $T>0$ and a complete filtered probability space $\left(\Omega, \left(\mathfrak F_t\right)_{t\in[0,T]}, \mathfrak F=\mathfrak F_T, \mathbb{P}\right)$. We consider a $d$-dimensional fractional Brownian motion $(B_t)_{t\in [0, T]}=(B^1_t,\ldots, B^d_t)_{t\in [0, T]}$ defined on $\Omega$.  We assume that the filtration $\left(\mathfrak F_t\right)_{t\in[0,T]}$ is generated by $B$.

Let $\mathcal B$ be the Gaussian space spanned by the (one-dimensional) fractional Brownian motion $B^{1}$. Let $\mathcal E$ be the linear space of step functions over $[0,T]$ and let $\mathcal H$ be the Hilbert space obtained as the completion of $\mathcal E$ with respect to the inner product induced from $B^{1}$:
\begin{equation}\langle\mathbb{I}_{[0,t]},\mathbb{I}_{[0,s]}\rangle_\mathcal{H}=\mathbb{E}[B^{1}_tB^{1}_s], \quad 0\leq s,t\leq T.\notag\end{equation}
The linear map defined on $\mathcal E$ by $\Phi: \mathbb{I}_{[0,t]}\rightarrow B^1_t$ is  an isometry from $(\mathcal E, \langle.,.\rangle_{\mathcal H})$ to $(\mathcal B, \mathbb{E}[.,.])$, and can thus be extended to an isometry from the whole space $\mathcal H$.

For $H=\frac{1}{2}$, we have $\mathcal H=L^2([0,T])$. When $H>\frac{1}{2}$, it is well known that $\mathcal H$ contains distributions, and therefore is not a subspace of some convenient functional space, see \cite{pipiras2001classes}. This is why we introduce the subspace $|\mathcal{H}|$ of $\mathcal H$, which is defined as the set of measurable functions $f:[0,T]\to\R$ such that $\int_{[0,T]^2}|f(x)||f(y)||x-y|^{2H-2}dxdy<+\infty$. From \cite{pipiras2001classes}, we have that $(|\mathcal H|,\|\cdot\|_{|\mathcal H|})$ is a Banach space with respect to the norm $\|\cdot\|_{|\mathcal{H}|}$, defined as
$$
\|f\|^2_{|\mathcal{H}|}=c_H\int_{[0,T]^2}|f(x)||f(y)||x-y|^{2H-2}dxdy,$$
with $c_H=H(2H-1)$.
We observe that $\|f\|_{|\mathcal{H}|}\leq \|f\|_{\mathcal{H}}$ for all $f\in |\mathcal{H}|$.

Still for $H>\frac12$, we define $|\mathcal H|^{\otimes p}$, $p\in\mathbb{N^*}$, to be the Banach space of measurable functions $f:[0,T]^p\to\R$ such that
\begin{equation}
\int_{[0,T]^{2p}}|f(x_1,\ldots,x_p)||f(y_1,\ldots,y_p)|\prod\limits_{i=1}^p|x_i-y_i|^{2H-2}dx_idy_i<+\infty.
\notag\end{equation}

Let $n\in\mathbb{N^*}$ and let $\mathcal S_n$ be the space of infinitely differentiable functions
$f:\mathbb{R}^{nd}\rightarrow \mathbb{R}$ such that $f$ and all its derivatives have at most polynomial growth. 
 We consider the \textit{Schwartz space} $\mathcal C$ composed of all cylindrical random variables, that is, of all random variables $F$ of the form
$F=f(B_{t_1},\ldots,B_{t_n})$, with $n\in\mathbb{N^*}$, $f\in\mathcal S_n$, and $t_1,\ldots,t_n\in[0,T]$.

The $p$th-order Malliavin derivative of $F\in \mathcal{C}$  is the element 
$$D^pF=\{D^{p,j_1,\ldots,j_p}_{l_1,\ldots,l_p}F:\,l_1,\ldots,l_p\in[0,T]\}_{1\leq j_1,\ldots,j_p\leq d}$$ 
belonging to $\cap_{r\geq 1}L^r(\Omega,\mathcal{H}^{\otimes p^2\times d})$ defined as
\begin{equation}D^{p,j_1,\ldots,j_p}_{l_1,\ldots,l_p}F=\sum_{k_1,\ldots,k_p=1}^n\frac{\partial^pf}{\partial x^{k_1,j_1}\ldots\partial x^{k_p,j_p}}(B_{t_1},\ldots,B_{t_n})\prod_{i=1}^p\mathbb{I}_{[0,t_{k_i}]}(l_i).
\notag\end{equation}

Since these operators are closable in $L^r(\Omega,\mathcal{H}^{\otimes p^2\times d})$ for all $r\geq 1$, we can consider the \textit{Sobolev space} $\mathbb{D}^{p,r}$ as the closure of $\mathcal C$ with respect to the norm 
\begin{equation}
\|F\|^r_{\mathbb{D}^{p,r}}=\mathbb{E}[|F|^r]+\sum\limits_{m=1}^p\sum_{j_1,\ldots,j_m=1}^d\mathbb{E}\left[\|D^{m,j_1,\ldots,j_m}F\|_{\mathcal H^{\otimes m}}^r\right].\notag\end{equation}

In the same way, it is possible to define the Malliavin derivative for step processes $u$ of the form
$u=\sum\limits_{i=0}^{n-1}F_i\mathbb{I}_{[t_i,t_{i+1}]}$
(where $n\in\mathbb{N^*}$, $t_0=0,t_1,\ldots,t_n\in [0,T]$ and $F_1,\ldots,F_n\in \mathcal C$),
and to consider the associated spaces $\mathbb{D}^{p,r}(\mathcal H)$. 
In order to only deal with functions (and not distributions), we consider the subspace $\mathbb{D}^{p,r}(|\mathcal H|)$ of $\mathbb{D}^{p,r}(\mathcal H)$, which
is by definition the set of 
$u\in \mathbb{D}^{p,r}(\mathcal H)$ that are such that $u\in |\mathcal H|$ a.s.\!, $D^1u\in (|\mathcal H|^{\otimes 2})^d$ a.s.\!,  $\ldots$, $D^pu\in (|\mathcal H|^{\otimes p+1})^{pd}$ a.s.\!. 
This subspace is endowed with the norm 
\begin{equation}
\|u\|^r_{\mathbb{D}^{p,r}(|\mathcal H|)}
=\mathbb{E}\|u\|^r_{|\mathcal H|}+
\sum\limits_{m=1}^p\sum_{j_1,\ldots,j_m=1}^d\mathbb{E}\left[\|D^{m,j_1,\ldots,j_m}u\|_{|\mathcal H|^{\otimes m}}^r\right].\notag
\end{equation}

Let $u\in L^2(\Omega,\mathcal H)$ be such that $|\mathbb{E}[\langle D^{1,j}F,u\rangle_{\mathcal H}]|\leq K_{u}\sqrt{\mathbb{E}[F^2]}$ for all $F\in \mathcal C$ and $j\in\{1,\ldots,d\}$, for some constant $K_{u}$ 
depending only on $u$.  
We then say that $u$ belongs to the domain ${\rm Dom}(\delta^{1,j})$, and we define the Skorohod integral $\delta^{1,j}$ as the adjoint of $D^{1,j}$, that is, $\delta^{1,j}(u)$ is the uniquely determined random variable in $L^2(\Omega)$  verifying the duality relationship: 
\begin{equation}
\mathbb{E}[\langle D^{1,j}F,u \rangle_{\mathcal H}]=\mathbb{E}[F\delta^{1,j}(u)]  \mbox{ for all } F\in\mathbb{D}^{1,2}.
\end{equation}

In the same way, if $u$ is an element of $L^2(\Omega, \mathcal H^{\otimes p})$ ($p\geq 2$) we define the operator $\delta^p=
(\delta^{p,j_1,\ldots,j_p})_{1\leq j_1,\ldots,j_p\leq d}$ as the adjoint of $D^p=(D^{p,j_1,\ldots,j_p})_{1\leq j_1,\ldots,j_p\leq d}$ through the identity:
\begin{equation}
\mathbb{E}[\langle D^{p,j_1,\ldots,j_p}F,u\rangle_{\mathcal H^{\otimes p}}]=\mathbb{E}[F\delta^{p,j_1,\ldots,j_p}(u)]  \mbox{ for all } F\in\mathbb{D}^{p,2}.
\notag\end{equation}
We can show that $\mathbb{D}^{p,2}(\mathcal H)\subset {\rm Dom}(\delta^p)$.  

The following two  results will be also useful. The first one is a straightforward consequence of the Hardy-Littlewood-Sobolev inequality (see \cite[Theorem 6]{beckner}), whereas the second one corresponds to 
\cite[Proposition 1.3.1]{nualart2006malliavin}.
\begin{proposition}\label{boundsSkorokhod}
\begin{enumerate}
\item Fix an integer $k\geq 1$. There exists $M>0$ such that, for all $u\in L^2(\Omega,L^2([0,T]^k))$,
\begin{equation}
\mathbb{E}\left[\|u\|_{|\mathcal H|^{\otimes k}}^2\right]\leq M\mathbb{E}\left[\|u\|^2_{L^2([0,T]^k)}\right].
\label{embed}
\end{equation}
\item For all $u,v\in \mathbb{D}^{1,2}(\mathcal H)$ and $j\in\{1,\ldots,d\}$, we have
\begin{equation}
\mathbb{E}[\delta^{1,j}(u)\delta^{1,j}(v)]= \mathbb{E}\left[\langle u,v\rangle_{\mathcal H}\right]+\mathbb{E}\left[\langle D^{1,j}_\cdot u_{\cdot\cdot}, D^{1,j}_{\cdot} v_{\cdot\cdot}\rangle_{\mathcal H\otimes\mathcal H}\right].
\label{Meyer}
\end{equation}
\end{enumerate}
\end{proposition}

\subsection{Multiple Wiener-It\^o integrals}
Throughout all this section, we assume for simplicity that the underlying fractional Brownian motion is one-dimensional, i.e. that $d=1$.
We write $D^k$ (resp. $\delta^k$) instead of $D^{k,1,\ldots,1}$ (resp. $\delta^{k,1,\ldots,1}$).

When the process $u$ is {\it deterministic} in $\mathcal H^{\otimes{k}}$, its Skorohod integral $\delta^k(u)$ is called the \textit{$k$th-order Wiener-It\^o integral} of $u$. 
If $\widetilde u$ denotes the symmetrization of $u$ (see the footnote \footnote{If $\{e_j\}_{j\geq 1}$ denotes an orthonormal basis of $\mathcal H$ and if $u$ is given by $u=\sum_{j_1,\ldots,j_k\geq 1}a_{j_1,\ldots,j_k}e_{j_1}\otimes\ldots\otimes e_{j_k}$, then $\widetilde u=\frac{1}{k!} \sum_\sigma \sum_{j_1,\ldots,j_k\geq 1}a_{j_1,\ldots,j_k}e_{j_{\sigma(1)}}\otimes\ldots\otimes e_{j_{\sigma(k)}}$, where the first sum runs over all permutation $\sigma$ of $\{1,\ldots,k\}$.}), we have $\delta^k(u)=\delta^k(\widetilde u)$; we can therefore assume without loss of generality that $u$ is symmetric.
In what follows, we denote by $\mathcal H^{\odot{k}}$ the set of symmetric elements in $\mathcal H^{\otimes{k}}$.

The following statement summarizes what is needed about multiple Wiener-It\^o integrals in this paper.
We refer e.g.\!\! to \cite{nourdin2012normal} for the proofs.

\begin{proposition}\label{multipleWiener}
\begin{enumerate}
\item (Isometry) For all integers $k,l\geq 1$, all $f\in\mathcal{H}^{\odot k}$ and all $g\in\mathcal{H}^{\odot l} $,
\begin{equation*}
\mathbb{E}[\delta^k(f)\delta^l(g)]=k!\langle f,g\rangle_{\mathcal{H}^{\otimes k}}\mathbb{I}_{\{k=l\}}.\end{equation*}
\item (Hypercontractivity) For all $r\geq 2$ and all integer $k\geq 1$, there exists $C_{k,r}>0$ such that, for all $f\in\mathcal{H}^{\odot k}$,
\begin{equation*}
\mathbb{E}\left[|\delta^k(f)|^r\right]\leq C_{k,r}\mathbb{E}[|\delta^k(f)|^2]^\frac{r}{2}.
\end{equation*}
\item (Malliavin derivative) If $u_s=\delta^k(f(.,s))$ with $f\in \mathcal H^{\otimes (k+1)}$ symmetric in the $k$ first variables, then $u\in\mathbb{D}^{1,2}(\mathcal H)$, with
\begin{equation*}
D_su_t=k\delta^{k-1}(f(.,t,s)).
\end{equation*}
\item (Product formula) Fix $f\in\mathcal H^{\odot k}$ and $g\in\mathcal H^{\odot l}$ and, as usual, let $\otimes_r$ (resp. $\widetilde{\otimes_r}$) denote the contraction operator (resp. the symmetrization of the contraction operator) of order $r$, see \cite[Appendix B]{nourdin2012normal} for a precise definition.
 Then,
\begin{equation*}
\delta^k(f)\delta^l(g)=\sum_{r=0}^{k\wedge l}r!\binom{k}{r}\binom{l}{r}f\delta^{k+l-2r}(f\widetilde{\otimes_r}g).
\label{product}
\end{equation*}
\end{enumerate}
\end{proposition}

\subsection{Fractional Integration}
This section gives a brief summary of the useful properties related to the Young integral when the Hurst index $H$ is strictly bigger than $\frac12$,
see \cite{young1936inequality,zahle1998integration} for more details.

The following result extends the Riemann integral to a larger class of integrands and integrators.
For $p>0$, we use the classical notations $\mathcal C^{p-var}([0,T])$ to denote the space of functions $f:[0,T]\to\R$ with finite $p$-variations. It is well known that $\theta$-H\"older continuous functions have $\frac{1}{\theta}$-finite variations.

\begin{proposition}\label{integralYoung}
Suppose $p,q>0$ are such that $\frac{1}{p}+\frac{1}{q}>1$.
If $f\in\mathcal C^{p-var}([0,T])$ and $g\in \mathcal C^{q-var}([0,T])$ (with $g$ continuous), then the limit of Riemann sums 
$$\sum\limits_{k=0}^{n-1}f\left(\frac{kT}{n}\right)\left(g\left(\left(\frac{(k+1)T}{n}\vee a\right)\wedge b\right)-g\left(\left(\frac{kT}{n}\vee a\right)\wedge b\right)\right)$$ exists for all $0\leq a<b\leq T$, and is called the Young integral $\int_a^bfdg$ of $f$ against $g$. It is compatible in the sense that, if $0\leq a<c<d<b\leq T$, then $\int_c^dfdg=\int_a^bf\mathbb{I}_{[c,d]}dg$. Moreover, it satisfies the chain rule and the change of variable formula.\\
Moreover, if $f$ (resp $g$) are $\frac{1}{p}$-H\"older continuous (resp $\frac{1}{q}$-H\"older continuous), we have the Young-Loeve estimates: 
\begin{eqnarray*}
\left|\int_a^bfdg-f(a)(g(b)-g(a))\right|\leq c_{\mu,\beta}\|f\|_{\frac{1}{p}}\|g\|_{\frac{1}{q}}|b-a|^{\frac{1}{p}+\frac{1}{q}},\\
\left|\int_a^bfdg\right|\leq c_{\mu,\beta}\left(\|f\|_{\infty}\|g\|_{\frac{1}{q}}|b-a|^\frac{1}{q}+\|f\|_{\frac{1}{p}}\|g\|_{\frac{1}{q}}|b-a|^{\frac{1}{p}+\frac{1}{q}}\right),
\end{eqnarray*}
where $c_{\mu,\beta}$ is a constant depending only on $p$ and $q$.
\end{proposition}
When  $f:[0,T]^2\to\R$ is such that $f(t,t)=0$, we write $f\in\mathcal{C}^{\kappa}([0,T]^2)$ if
\begin{equation}\label{cc}
\|f\|_{\kappa}:=\sup_{0\leq s\neq t\leq T}\frac{|f(s,t)|}{|t-s|^{\kappa}}<\infty.
\end{equation}

Recall that, for each $i$, the fractional Brownian motion $B^i$ has a.s.\!\! $\kappa$-H\"older continuous paths for every $\kappa<H$. Therefore, if the process $u$ has a.s.\! finite $\alpha$-variations for some $\alpha>1-H$, it is an immediate consequence of Proposition \ref{integralYoung}  that the Young integral $\int_0^\cdot udB^{i}$ is well-defined pathwise on $[0,T]$; this makes the Young integral a suitable integral when $H>\frac12$. In contrast, it is not a suitable integral when $H=\frac{1}{2}$ because, for instance, we cannot deal with integrals as simple as $\int B^jdB^i$. 

Another way to define the Young integral is to make use of the forward integration {\it \`a la Russo-Vallois}  
 \cite{russo1993forward}. Their {\it forward integral} is defined, for fixed $j$, as
\begin{equation}\int_0^\cdot u_sdB^j_s=\lim_{\epsilon\rightarrow 0}\frac{1}{\epsilon}\int_0^\cdot u_s\left(B^j_{s+\epsilon\wedge \cdot}-B^j_s\right)ds,\label{RussoVallois}\end{equation}
provided the limit exists uniformly in probability over the interval $[0,T]$. 
When $H>\frac{1}{2}$ and $u\in C^\theta([0,T])$ with $\theta>1-H$, then the limit \eqref{RussoVallois} exists and coincides with the Young integral. When $H=\frac{1}{2}$ and $u$ is progressively measurable, then the limit \eqref{RussoVallois} exists and coincides with the It\^o integral.

In \cite{nualart2006malliavin}, the following relationship between the forward and Skorohod integrals is shown.

\begin{proposition}\label{traceClass}
Assume that $H>\frac{1}{2}$, and let $u\in\mathbb{D}^{1,2}(|\mathcal H|)$ be a scalar process. In addition, suppose that $u$ verifies the following condition:
\begin{equation}\forall j\in \{1,\ldots,d\},\int_0^T\int_0^T|D^{1,j}_su_r| |r-s|^{2H-2}dsdr<\infty \text{ a.s.}.\label{asCond}\end{equation}
Then, the limit \eqref{RussoVallois} exists and verifies the relation
\begin{equation}
\int_0^Tu_sdB^{j}_s=\delta^{1,j}(u)+H(2H-1)\int_0^T\int_0^TD^{1,j}_su_r|r-s|^{2H-2}dsdr,
\label{fracMal}
\end{equation}
where the integral in the left-hand side is in the Russo-Vallois sense.
\end{proposition}

\subsection{Matrix-valued Brownian motion and matrix-valued Rosenblatt process}\label{rosenblatt-and-brownian}
We introduce some probabilistic objects, taken from \cite[Sections 2.4 and 2.5]{hu2016rate} when $H>\frac12$, which we complete when
$H=\frac{1}{2}$. 
For more information about the Rosenblatt process, one can e.g. refer to \cite{tudor2008analysis}.

$(a)$ Assume first that $H\in\left[\frac{1}{2},\frac{3}{4}\right]$. For $H\not\in \{\frac12,\frac34\}$, define 
\begin{eqnarray}
q_H&=&\sum_{p\in\mathbb{Z}}\int_0^1\int_0^t\int_{p}^{p+1}\int_p^v|s-u|^{2H-2}|v-t|^{2H-2}dsdvdudt\notag\\
r_H&=&\sum_{p\in\mathbb{Z}}\int_0^1\int_0^t\int_{p}^{p+1}\int_p^v|s-t|^{2H-2}|v-u|^{2H-2}dsdvdtdu\notag,
\end{eqnarray}
and let $q_{\frac{1}{2}}=\frac{1}{2}$, $r_\frac{1}{2}=0$ and $q_\frac{3}{4}=r_\frac{3}{4}=\frac{1}{2}$.
We have $q_H\geq r_H$ by \cite[Lemma 2.1]{hu2016rate}.
Let $\{W^{0,i,j}\}_{1\leq i\leq j\leq d}$ and
$\{W^{1,i,j}\}_{1\leq i, j\leq d}$  be two independent families of independent standard Brownian motions, both independent of our underlying process $B$. We set $W^{0,i,j}=W^{0,j,i}$ for $j<i$.
The {\it matrix-valued Brownian motion} $(W^{i,j})_{1\leq i,j\leq d}$ is then defined as follows: 
\begin{eqnarray}\label{wh}
W^{i,j}&=&
\left\{
\begin{array}{lll}
c_H\sqrt{q_H+r_H}\,W^{1,i,j}&\quad\mbox{if $i=j$}\\
c_H\sqrt{q_H-r_H}W^{1,i,j}+c_H\sqrt{r_H}W^{0,i,j}&\quad\mbox{if $i\neq j$}
\end{array}
\right.,
\end{eqnarray} 
with the convention that $c_{\frac{1}{2}}=1$.

\medskip
$(b)$ Assume now that $H\in(\frac{3}{4},1)$. For any fixed $t\in[0,T]$, the sequence of $(d\times d)$-matrix-valued processes \begin{equation}\left(n\sum_{k=0}^{\lfloor nt \rfloor-1}\delta^{1,i}\left(\left(B^{j}_.-B^{j}_\frac{k }{n}\right)\mathbb{I}_{\left[\frac{k}{n},\frac{k+1}{n}\right]}(.)\right)\right)_{1\leq i,j\leq d}\notag\end{equation}
converges for all fixed $t\in [0,T]$ to some $Z_t$. The continuous version of the process $(Z_t)_{t\in[0,T]}$ is called the {\it matrix-valued Rosenblatt proces of order H}. Each component of this matrix-valued process is $\alpha$-H\"older continuous for every $\alpha<2H-1$. Moreover, the diagonals elements are independent Rosenblatt processes with selfsimilarity index $2H-1$.

\section{Examples}

We start by defining the notion of controlled process. This notion plays a key role because such a process verifies the conditions of Definition \ref{defpseudo}. 
We then give two classes of examples: fractional semimartingales (i.e. processes with decomposition (\ref{integralProcess})) and multiple Wiener-It\^o integrals.

\subsection{Controlled process}\label{sec:Controlled paths}
Throughout all this section, we assume $H>\frac{1}{2}$.\\
Recall that $\mathcal{C}^{\kappa}([0,T])$ denotes the set of $\kappa$-H\"older continuous functions $f:[0,T]\to\R$,
whereas $\mathcal{C}^{\kappa}([0,T]^2)$ denotes the set of $\kappa$-H\"older continuous functions $f:[0,T]^2\to\R$ such that $f(t,t)=0$ for all $t$, see (\ref{cc}).

\begin{definition}[Controlled process]\label{def2}
\textit{Consider $\kappa\in(\frac12,1)$.
The set $\mathcal{D}^{2\kappa}([0,T])$ is defined as the set of pairs 
$(u,P)$ with $u$ (resp. $P$) an $m$-dimensional process 
(resp. $(m\times d)$-dimensional process) belonging a.s.\! to
$ \mathcal{C}^{\kappa}([0,T])$ and such that the $m$-dimensional remainder process $R$ defined by
\begin{equation}
R^i_{s,t} = u^i_t-u^i_s-\sum_{j=1}^dP^{i,j}_s(B^j_t-B^j_s), \quad 0\leq s\leq t\leq T, 
\label{controlled}\end{equation}
belongs a.s.\!\! to $\mathcal{C}^{2\kappa}([0,T]^2)$.}
\end{definition}

For all $(u,P)\in \mathcal{D}^{2\kappa}([0,T])$, for all $s,t\in [0,T]$ and all $j\in\{1,\ldots,d\}$, Theorem 4.10 in \cite{friz2014course} implies:
\begin{equation}|L^{i,j}_{s,t}|\leq C\left(\|B\|_{\kappa}\|R\|_{2\kappa}+\|P\|_{\kappa}\|\mathbb{B}\|_{2\kappa}\right)|t-s|^{3\kappa}\label{enhancedYoung}\end{equation}
where $\mathbb{B}$ is defined as
$\mathbb{B}_{s,t}^{k,j}=\int_s^t(B^k_l-B^k_s)dB_l^j$, $L$ is given by $L^{i,j}_{s,t}=\int_s^tR^i_{s,r}dB_r^j$ (or equivalently by (\ref{Lij})), and
$C$ is a constant depending only on $\kappa$ and $T$.
The following proposition gives an explicit link between the notion of 
controlled path {\it \`a la Gubinelli} \cite{gubinelli2004controlling} (Definition \ref{def2}) and our notion of pseudo-controlled path (Definition \ref{defpseudo}).

\begin{proposition}\label{controlledPaths}\textit{ 
Assume that $\kappa>\frac{2(H\wedge\frac{3}{4})}{3}+\frac{1}{6}$,
 $(u,P)\in\mathcal{D}^{2\kappa}([0,T])$ and, for some $\theta>0$  and all $j\in\{1,\cdots,d\}$,
\begin{equation}\sum_{i=1}^m\mathbb{E}\left[\|R^i\|^{2+\theta}_{2\kappa}+\sum_{j=1}^d\|P^{i,j}\|^{2+\theta}_{\kappa}\right]<\infty,
\label{moment}\end{equation}
with $R$ defined by (\ref{controlled}).
Then $(u,P)\in\mathbb{C}_2$.
}\end{proposition}
\begin{proof} The proof is a straightforward combination of the identity \eqref{enhancedYoung}, the H\"older inequality and the forthcoming Lemmas \ref{Fernique} and \ref{simpleInequality}. 
\end{proof}

As a consequence of Theorem \ref{main2} and Proposition \ref{controlledPaths}, we deduce the following statement.
\begin{proposition}\label{integralProcesses}
Fix $H>\frac{1}{2}$, and let 
\begin{equation}
u^i_t=u^i_0+\sum_{j=1}^d\int_0^ta^{i,j}_sdB^j_s+\int_0^tb^i_sds,\quad t\in[0,T],\,i\in\{1,\ldots,m\},
\label{integralProcess}
\end{equation}
where the $a^{i,j}$ are a.s. \!\!$\kappa$-H\"older continuous for some $\kappa>\frac{2(H\wedge\frac{3}{4})}{3}+\frac{1}{6}$ and the $b^j$ are $\beta$-H\"older continuous for some $\beta>H-\frac{1}{2}$.
Assume moreover that  there exists $\theta>0$ such that $$\sum_{j=1}^d \mathbb{E}\left[
|b^j_0|^{2+\theta}+\|b^j\|_\beta^{2+\theta}+\sum_{i=1}^m\|a^{i,j}\|^{2+\theta}_{\kappa}
\right]<\infty.$$ 
Then, with $M^{n,i,j}$ defined by (\ref{mnij})
and $W$ and $Z$ the matrix-valued processes of Section \ref{rosenblatt-and-brownian},
\begin{itemize}
\item if $H\leq \frac{3}{4}$, then, stably in $\mathcal{C}_{\mathbb{R}^{m\times d}}([0,T])$,
\begin{equation*}\left\{\nu_H(n)\left(M_\cdot^{n,i,j}-\frac{1}{2}\int_0^\cdot a^{i,j}_sds\right)\right\}_{i,j}\overset{}{\underset{n\rightarrow\infty}\longrightarrow}\left\{\sum_{k=1}^{d}\int_0^\cdot a_s^{i,k}dW^{k,j}_s\right\}_{i,j}. 
\end{equation*}
\item if $H>\frac{3}{4}$, then, uniformly on $[0,T]$ in probability,
\begin{equation*}
\left\{\nu_H(n)\left(M_\cdot^{n,i,j}-\frac{1}{2}\int_0^\cdot a^{i,j}_sds\right)\right\}_{i,j}\overset{}{\underset{n\rightarrow\infty}\longrightarrow}\left\{\sum_{k=1}^{d}\int_0^\cdot a_s^{i,k}dZ^{k,j}_s+\frac{1}{2}\int_0^\cdot b^i_sdB^j_s\right\}_{i,j}.
\end{equation*}
\end{itemize}
\end{proposition}
\begin{proof} 
Set $v^i_t=u^i_t-\int_0^tb^i_sds=u^i_0+ \sum_{j=1}^d\int_0^ta^{i,j}_sdB^j_s$. For any $i,j$, we have
\begin{equation*}
\nu_H(n)\left(M_\cdot^{n,i,j}-\frac{1}{2}\int_0^\cdot a^{i,j}_sds\right)=A^{n,i,j}_\cdot+C^{n,i,j}_\cdot
\end{equation*}
with,  for $t\in [0,T]$,
\begin{eqnarray*}
A^{n,i,j}_t&=&\nu_H(n)\left\{n^{2H-1}\left(\int_0^t v^i_sdB^j_s-\sum_{k=0}^{nt_n}v^i_\frac{k}{n}\left(B^j_{\frac{k+1}{n}\wedge t}-B^j_\frac{k}{n}\right)\right)\right.\\
&&\hskip8cm \left. -\frac{1}{2}\int_0^t a^{i,j}_sds\right\}\\
C^{n,i,j}_t&=&\nu_H(n)n^{2H-1}\left(\int_0^t\int_0^s b^i_rdrdB^j_s-\sum_{k=0}^{nt_n}\int_{0}^{\frac{k}{n}}b^i_rdr\left(B^j_{\frac{k+1}{n}\wedge t}-B^j_\frac{k}{n}\right)\right).
\end{eqnarray*}
We will show that $(v,a)\in\mathbb{C}_2$ and we will deduce from Theorem \ref{main2}  the convergence of $(A^{n,i,j})_{i,j}$.
Then we will prove that
$(C^{n,i,j})_{i,j}$
converges either to $0$ in $\mathcal{C}^{m\times d}([0,T])$ (when $H\leq\frac{3}{4}$) or uniformly in probability to $\frac{1}{2}\int_0^\cdot b_sdB_s$ (when $H>\frac34$). The continuous mapping theorem will then allow to conclude.

We start by showing that $(v,a)\in \mathbb{C}_2$.
For $0\leq s\leq t\leq T$, set
\begin{equation*}
R^i_{s,t}=v^i_t-v^i_s-\sum_{j=1}^d a^{i,j}_s(B^j_t-B^j_s)=\sum_{j=1}^d\int_s^t(a^{i,j}_r-a^{i,j}_s)dB^j_r.
\end{equation*}
Using the Young-Loeve inequality (Proposition \ref{integralYoung}), we have
\begin{eqnarray*}
|R^i_{s,t}|&\leq& |t-s|^{2\kappa}\times c_{\kappa}\sum_{j=1}^d \|a^{i,j}\|_{\kappa}\|B^j\|_{\kappa}\\
\|v^i\|_{\kappa}&\leq& \sum_{j=1}^d\left(\|a^{i,j}\|_{\infty}\|B^j\|_{\kappa}+c_{\kappa,\kappa}\|a^{i,j}\|_{\kappa}\|B^j\|_{\kappa}T^{\kappa}\right)\\
&\leq& \sum_{j=1}^d\left((1+c_{\kappa,\kappa})T^{\kappa}\|a^{i,j}\|_{\kappa}+|a^{i,j}_0|\right)\|B^j\|_{\kappa}
\end{eqnarray*}
where the last inequality comes from the fact that $\|a^{i,j}\|_{\infty}\leq |a^{i,j}_0|+T^{\kappa}\|a^{i,j}\|_{\kappa}$.
Thus, $v$ verifies the condition of Proposition \ref{controlledPaths}, with $P^{i,j}=a^{i,j}$. We deduce that $(v,a)\in\mathbb{C}_2$, and we can apply  Theorem \ref{main2}
to $(v,a)$, after observing that $v$ is $\alpha$-H\"older continuous for all $\alpha=\kappa>\frac12>1-H$. This shows the convergence of $(A^{n,i,j})_{i,j}$.

We now study the convergence of $C^{n,i,j}$. Set $s_n=\lfloor ns\rfloor/n$. We have
\begin{eqnarray*}
C_t^{n,i,j}
&=&\nu_H(n)n^{2H-1}\left(\int_0^t\int_0^sb^i_rdrdB^j_s-\int_0^{t}\int_0^{s_n}b^i_rdrdB^j_s\right)\\
&=&\nu_H(n)n^{2H-1}\left(\int_0^t\int_{s_n}^s\left(b^i_r-b^i_{s_n}\right)drdB^j_s\right.\\
&&\hskip3.5cm\left.+\sum_{k=0}^{nt_n}b^i_\frac{k}{n}\int_{\frac{k}{n}}^{\frac{k+1}{n}\wedge t}(s-s_n)dB^j_s\right)\\
&=:&R_t^{n,i,j}+D_t^{n,i,j}.
\end{eqnarray*} 
Lemma \ref{driftConvergence} provides the desired convergence for $D^{n,i,j}$. It remains to show that $R^{n,i,j}$ is negligible.
We have
$$
R_t^{n,i,j} = \nu_H(n)n^{2H-1} \sum_{k=0}^{nt_n}\int_{\frac{k}{n}}^{\frac{k+1}{n}\wedge t}\left(
\int_{\frac{k}n}^s (b^i_r-b^i_{\frac{k}n})dr
\right)dB^j_s.
$$
Fix $\varepsilon>0$ small enough. We can write, using the Young-Loeve inequalities (Proposition \ref{integralYoung}) and denoting by $c$ a constant independent of $n$ (whose value can change from line to another)
\begin{eqnarray*}
&&\int_{\frac{k}{n}}^{\frac{k+1}{n}\wedge t}\left(
\int_{\frac{k}n}^s (b^i_r-b^i_{\frac{k}n})dr
\right)dB^j_s\\
&\leq&c\,n^{-H+\varepsilon} \,
\left\|
\int_{\frac{k}{n}}^{\cdot}  (b^i_r-b^i_{\frac{k}n})dr
\right\|_{\infty,\big[\frac{k}{n},\frac{k+1}{n}\wedge t\big]}
\|B^j\|_{H-\varepsilon}
\\
&+&
c\, n^{-1-H+\varepsilon} \,
\left\|
\int_{\frac{k}{n}}^{\cdot}  (b^i_r-b^i_{\frac{k}n})dr
\right\|_{1,\big[\frac{k}{n},\frac{k+1}{n}\wedge t\big]}
\|B^j\|_{H-\varepsilon}\\
&\leq& c\,n^{-1-H-\beta+\varepsilon}\, \|b^i\|_{\beta}\,\|B^j\|_{H-\varepsilon}.
\end{eqnarray*}
We deduce that
$$
\big|R_t^{n,i,j}\big| \leq c\, \nu_H(n) n^{H-1-\beta+\varepsilon}\, \|b^i\|_{\beta}\,\|B^j\|_{H-\varepsilon},
$$
and then $\EE\left[\sup_{t\in [0,T]}(R_t^{n,i,j})^2\right]\rightarrow 0$ (chosing $\varepsilon$ small enough), proving the convergence of this remainder to zero uniformly in probability.
This concludes the proof of Proposition \ref{integralProcesses}.
\end{proof}

We now state a corollary of Proposition \ref{integralProcesses}, which extends to the case $H>\frac12$ a similar statement proved in \cite{lindberg2013error} when $H=\frac{1}{2}$.

\begin{corollary}\label{functionB}
\textit{ 
Fix $H>\frac{1}{2}$, and
let $F:\R^d\to\R^m$ be a $\mathcal{C}^2$-function satisfying the following growth condition:
for some $K_1,K_2>0$ and some $0<\gamma<2$, one has,   for all $x\in\R^d$,
\begin{equation}
\max_{i\in\{1,\ldots,m\}}\max_{j,k\in\{1,\ldots,d\}}\max\left\{|F^i(x)|,\left\lvert\frac{\partial F^i}{\partial x_j}(x)\right\rvert, \left\lvert\frac{\partial^2 F^i}{\partial x_k\partial x_j}\right\rvert\right\}\leq K_1e^{K_2\|x\|_{\R^d}^{\gamma}}.
\label{regCond}
\end{equation}
Let $u_t=F(B_t)$. We have, with $W$ and $Z$ the matrix-valued processes of Section \ref{rosenblatt-and-brownian}:
\begin{itemize}
\item if  $H\leq \frac34$, then, stably in $\mathcal{C}_{\mathbb{R}^{m\times d}}([0,T])$,
\begin{equation*}
\left\{\nu_H(n)\left(M_\cdot^{n,i,j}-\frac{1}{2}\int_0^\cdot\frac{\partial F^i}{\partial x_j}(B_s)ds\right)\right\}_{i,j}\overset{}{\underset{n\rightarrow\infty}\longrightarrow}\left\{\sum_{k=1}^{d}\int_0^.\frac{\partial F^i}{\partial x_k}(B_s)dW^{k,j}_s\right\}_{i,j};
\end{equation*}
\item if $H>\frac{3}{4}$, then, uniformly on $[0,T]$ in probability,
\begin{equation*}
\left\{\nu_H(n)\left(M_\cdot^{n,i,j}-\frac{1}{2}\int_0^\cdot\frac{\partial F^i}{\partial x_j}(B_s)ds\right)\right\}_{i,j}\overset{}{\underset{n\rightarrow\infty}\longrightarrow}\left\{\sum_{k=1}^{d}\int_0^.\frac{\partial F^i}{\partial x_k}(B_s)dZ^{k,j}_s\right\}_{i,j}.
\end{equation*}
\end{itemize}
}\end{corollary}
\begin{proof}The change of variable formula for the Young integral leads to
\begin{equation}u^i_t=F^i(0)+\sum_{j=1}^d\int_0^t\frac{\partial F^i}{\partial x_j}(B_s)dB^{j}_s, \quad 1\leq i\leq m.\notag\end{equation}
Then, $u$ is of the type (\ref{integralProcess}), with $a^{i,j}_\cdot=\frac{\partial F^i}{\partial x_j}(B_\cdot)$ and $b^i\equiv 0$.
The regularity condition \eqref{regCond} implies that $ a^{i,j}$ is $\alpha$-H\"older continous for every $\alpha<H$ and that \begin{equation}\|a^{i,j}\|_{\alpha}\leq K_1\prod_{j=1}^{d}e^{K_2T^\gamma (\|B^{j}\|_{\alpha})^\gamma}\sum_{k=1}^{d}\|B^{k}\|_{\alpha}.\notag\end{equation}
Lemma \ref{Fernique} then guarantees the existence of moments of any order for this random variable, so that the desired conclusion follows from Proposition \ref{integralProcesses}.
\end{proof}

\subsection{Multiple Wiener-It\^o integrals}\label{MWI}
Assume $H>\frac{1}{2}$ and, for simplicity, $d=m=1$.
Let $k\geq 1$ be an integer and let $f_k:[0,T]^{k+1}\rightarrow\mathbb{R}$ be measurable and symmetric in the  first $k$ variables (this latter condition is of course immaterial when $k=1$). 
Assume finally that $f_k(x_1,\ldots,x_k,s)=0$ if $x_l>s$ for at least one $l$. In that setting, Theorems \ref{main1} and \ref{main2} apply.

\begin{proposition}\label{convergenceWienerIto}\textit{ 
Let the previous notation prevail, as well as the notation from Section 2.2.
\begin{enumerate}
\item
Assume that $f_k$ is $\alpha$-H\"older continuous on $$\mathfrak D=\{(x_1,\ldots,x_k,s)\in [0,T]^{k+1}, s\geq \max(x_1,\ldots,x_k)\},$$ for some $\alpha>H$. 
Set $u_s=\delta^k\left(f_k(\cdot,s)\right)$. Then,
uniformly on $[0,T]$ in probability,
\begin{equation}
M_\cdot^{n}\overset{}{\underset{n\rightarrow\infty}{\longrightarrow}}\frac{k}{2}\int_0^\cdot\delta^{k-1}(f_k(.,s,s))ds.\notag\end{equation}
\item Assume $\frac{1}{2}<H\leq \frac{3}{4}$. Assume moreover that the hypothesis of the previous point holds, and that in addition $$f_k(x_1,\ldots,x_k,s)=g_k(x_1,\ldots,x_k)\mathbb{I}_{[0,s]^k}(x_1,\ldots,x_k)$$ with $g_k$ symmetric and $\beta$-H\"older continuous for some $\beta>\frac{1}{2}$. Then,
stably in $\mathcal{C}_{\mathbb{R}}([0,T])$ and with $W$ an independent standard Brownian motion,
\begin{eqnarray}
&&\nu_H(n)\left(M_\cdot^{n}-\frac{k}{2}\int_0^\cdot\delta^{k-1}(f_k(\cdot,s,s))ds\right)\notag\\
&\overset{}{\underset{n\rightarrow\infty}\longrightarrow}&
c_H\sqrt{q_H+r_H}
\int_0^\cdot\delta^{k-1}(f_k(\cdot,s,s))dW_{s}, 
\notag
\end{eqnarray}
where $q_H$ and $r_H$ as defined in Section \ref{rosenblatt-and-brownian}.
\end{enumerate}
}
\end{proposition}
\noindent\textit{Proof.} We only do the proof of point (2), since the proof of point $(1)$ (which requires to show that $(u,P)$ with $P_s:=D_su_s$ verifies the assumptions of Theorem \ref{main1}) is very similar and easier. 
Before going into the details, let us explain the main steps we are going to follow:
\begin{itemize}
\item in the first step, we show that $u$ and $P$ are $\beta'$-H\"older continuous for some $\beta'>\frac{1}{2}>1-H$;
\item in the second step, we provide a suitable decomposition of $L_{s,t}L_{x,y}$.  We recall that $L_{s,t}$ is defined as \begin{equation}\label{*}L_{s,t}=\int_{s}^t\left(u_l-u_s-D_su_s(B_l-B_s)\right)dB_l;\end{equation}
\item finally, in the remaining steps, we analyze each term of the previous decomposition and show that the stuctural condition (\ref{pseudo}) is verified, i.e,
for all $0\leq s\leq t\leq T$ and all $0\leq x\leq y\leq T$,
\begin{equation}\label{remindPseudo}
 \mathbb{E}\left[L_{s,t}L_{x,y}\right]=o_{|s-t|+|x-y|\rightarrow 0}(f_2(s,t,x,y))\quad\mbox{uniformly in $s,t\in[0,T]$}.
\end{equation}
\end{itemize}

\noindent
\textit{{\underline{Step 1: H\"older continuity}}}.
The process $u$ is adapted with respect to $B$ and belongs to $\mathbb{D}^{1,2}(|\mathcal{H}|)$ with $D_su_t=k\delta^{k-1}\left(f_k(.,s,t)\right)\mathbb{I}_{s\leq t}$ by Propostion \ref{multipleWiener}. Using the hypercontractivity and isometry properties (again Proposition \ref{multipleWiener}), we obtain, for $a>1$ and $s\leq t$,
\begin{eqnarray*}
\mathbb{E}[|u_s-u_t|^a]&\leq& C_{k,a}\mathbb{E}[(u_s-u_t)^2]^{\frac{a}{2}}\\
&=& C_{k,a}\| f_k(\cdot,s)-f_k(\cdot,t)\|^a_{\mathcal H^{\otimes k}}\\
&\leq&C_{k,a}\| f_k(\cdot,s)-f_k(\cdot,t)\|^a_{|\mathcal H|^{\otimes k}},
\end{eqnarray*} 
thanks to the continuous embedding $|\mathcal H|^{\otimes k}\subset \mathcal H^{\otimes k}$ in the last line.

Let $\Delta_{s,t} f_k(\cdot)=f_k(\cdot,t)-f_k(\cdot,s)$.
We have
\begin{eqnarray*} 
&&\| f_k(\cdot,s)-f_k(\cdot,t)\|^2_{|\mathcal H|^{\otimes k}} \\
&=& c^k_H\int_{[0,t]^{2k}}|\Delta_{s,t} f_k(x)||\Delta_{s,t} f_k(y)|\prod_{m=1}^{k}|x_m-y_m|^{2H-2}dx_mdy_m \\ 
&=& c^k_H\sum_{i,j=1}^k
\int_{[0,t]^{2k}}
\mathbb{I}_{[s,t]}(x_i)\mathbb{I}_{[s,t]}(y_j)
|g_k(x)||g_k(y)|
\prod_{m=1}^{k}|x_m-y_m|^{2H-2}dx_mdy_m \\
&=& c^k_H\sum_{\substack{i,j=1\\i\neq j}}^k
\int_{[0,t]^{2k}}
\mathbb{I}_{[s,t]}(x_i)\mathbb{I}_{[s,t]}(y_j)
|g_k(x)||g_k(y)|\\
&&\times\left(\prod_{\substack{m=1\\m\neq i,j}}^k|x_m-y_m|^{2H-2}dx_mdy_m\right)(|x_i-y_i||x_j-y_j|)^{2H-2}dx_idy_idx_jdy_j\\
&&+c^{k}_H\sum_{i=1}^k
\int_{[0,t]^{2k}}
\mathbb{I}_{[s,t]}(x_i)\mathbb{I}_{[s,t]}(y_i)
|g_k(x)||g_k(y)|\\
&&\times\left(\prod_{\substack{m=1\\m\neq i}}^k|x_m-y_m|^{2H-2}dx_mdy_m\right)
|x_i-y_i|^{2H-2}dx_idy_i.
\end{eqnarray*}
From Lemma \ref{simpleInequality}), we have 
$$\int_{[0,t]\times[s,t]}|x-y|^{2H-2}dxdy\leq K|t-s|$$ for some constant $K$.
Note that we take the liberty to change the value of $K$ from line to line in the rest of the proof. 
We deduce, for $i\neq j$, that
\begin{eqnarray*}
&&c_H^k\int_{[0,t]^{2k}}
\mathbb{I}_{[s,t]}(x_i)\mathbb{I}_{[s,t]}(y_j)
|g_k(x)||g_k(y)|
\prod_{\substack{m=1\\m\neq i,j}}^k|x_m-y_m|^{2H-2}dx_mdy_m\\
&\leq& \|g_k\|^2_{\infty}|t-s|^2c_H^{k-2}\int_{[0,t]^{2k-2}}
\prod_{\substack{m=1\\m\neq i,j}}^k|x_m-y_m|^{2H-2}dx_mdy_m.
\end{eqnarray*}
As a result,
\begin{equation*}
\| f_k(\cdot,s)-f_k(\cdot,t)\|^2_{|\mathcal H|^{\otimes k}}\leq K\|g_k\|_{\infty}^2\big( (k-1)^2|t-s|^2t^{2H(k-2)}+k|t-s|^{2H}t^{2H(k-1)}\big).
\end{equation*}
Since $|t-s|\leq K|t-s|^{H}$ on $[0,T]^2$, this leads to
\begin{equation*}
\mathbb{E}[|u_s-u_t|^a]\leq K\|g_k\|^a_{\infty}|t-s|^{Ha}.
\end{equation*}
We can show a similar bound for the derivative $Du$:
\begin{eqnarray*}
&&\mathbb{E}[|D_su_s-D_tu_t|^a]\leq\mathbb{E}[|D_su_s-D_su_t|^a]+\mathbb{E}[|D_su_t-D_tu_t|^a]\\&\leq& C_{k-1,a}\left(\| f_k(\cdot,s,s)-f_k(\cdot,s,t)\|^a_{|\mathcal H|^{\otimes k-1}}+\| f_k(\cdot,s,t)-f_k(\cdot,t,t)\|^a_{|\mathcal H|^{\otimes k-1}}\right)
\end{eqnarray*}
and 
\begin{eqnarray*} 
&&\| f_k(\cdot,s,s)-f_k(\cdot,s,t)\|^2_{|\mathcal H|^{\otimes k-1}}+\| f_k(\cdot,s,t)-f_k(\cdot,t,t)\|^2_{|\mathcal H|^{\otimes k-1}}\\
&\leq& K\|g_k\|^2_{\infty}(k|t-s|^{2H}t^{2H(k-1)}+(k-1)^2t^{2H(k-2)}|t-s|^2)\\
&&+t^{2Hk}\|g_k\|^2_{\beta}|t-s|^{2\beta},
\end{eqnarray*}
where $\|g_k\|_{\beta}$ is the H\"older seminorm of $g_k$ over $[0,T]^k$. Then,
\begin{equation*}
\mathbb{E}[|D_su_s-D_tu_t|^a]\leq K(\|g_k\|^a_{\infty}|t-s|^{Ha}+\|g_k\|^a_\beta|t-s|^{\beta a}).
\end{equation*}
 Finally, for all $a>1$ we have
\begin{eqnarray*}
\mathbb{E}[|u_s-u_t|^a+|D_su_s-D_tu_t|^a]&\leq& C\left(|t-s|^{aH}+|t-s|^{a\beta}\right)\\
&=&C\left(|t-s|^{a'+1}+|t-s|^{a''+1}\right),
\end{eqnarray*}
with $a'=aH-1, a''=a\beta-1$  and the constant $C$ depending on $k,a,\|g_k\|_{\infty}$ and $T$.\\
Observe that $\frac{a'}{a}\rightarrow H$ and $\frac{a''}{a}\rightarrow\beta$ when $a\rightarrow\infty$.
 The Kolmogorov-Censov  criterion applies and yields that $u$ and $s\rightarrow D_su_s$ verifies the H\"older seminorm condition in Theorem \ref{main2}, namely:
 $u$ and $P$ are $\beta'$-H\"older continuous for all $\beta'$ such that $\beta\wedge H>\beta'>\frac{1}{2}>1-H$.\\

\noindent\textit{{\underline{Step 2: Decomposition of $L_{s,t}L_{x,y}$}}}
(recall the definition of $L$ from (\ref{*})).
 The product formula \eqref{product} yields, for $s\leq t$,
\begin{eqnarray*}
k\delta^{k-1}\left(f_k(\cdot,s,s)\right)\left(B_t-B_s\right)&=&k\delta^k\left(\widetilde{f_k(\cdot,s,s)\otimes\mathbb{I}_{[s,t]}(\cdot\cdot)}\right)\\
&+&k(k-1)\delta^{k-2}\left(f_k(\cdot,s,s)\otimes_1\mathbb{I}_{[s,t]}\right).
\end{eqnarray*}
Then, $u_t-u_s-D_su_s(B_t-B_s)=A_{s,t}-C_{s,t}$, with
\begin{equation}\label{AandB}\left\{
      \begin{array}{cll}
        A_{s,t}&=&\delta^k\left(f_k(\cdot,t)-f_k(\cdot,s)-k\widetilde{f_k(\cdot,s,s)\otimes\mathbb{I}_{[s,t]}(\cdot\cdot)}\right)\\
        C_{s,t}&=&k(k-1)\delta^{k-2}\left(f_k(\cdot,s,s)\otimes_1\mathbb{I}_{[s,t]}\right).
      \end{array}
    \right.
    \end{equation}

Notice that $C_{s,t}=0$ when $k=1$. We also use the convention that $A_{s,t}=C_{s,t}=0$ if $s>t$.\\
One can see\footnote{Indeed, assuming $g_k=1$, i.e $f_k(x_1,\ldots,x_{k+1})=\mathbb{I}_{[0,x_{k+1}]^{k}}(x_1,\ldots,x_k)$ we can write
$C_{s,t}=\delta^{k-2}\left(\mathbb{I}_{[0,s]^{k-2}}(\cdot)\int_0^s\int_s^t|l-r|^{2H-2}dldr\right)=r^H(0,s,s,t)\delta^{k-2}(\mathbb{I}_{[0,s]^{k-2}}).$
Since $r^H(0,s,s,t)>s|t-s|$ thanks to Lemma \ref{simpleInequality}, we have $\left|\frac{C_{s,t}}{|t-s|^{2\kappa}}\right|\geq \frac{s|\delta^{k-2}(\mathbb{I}_{[0,s]^{k-2}})|}{|t-s|^{2\kappa-1}}$ for any $\kappa>\frac{1}{2}$. We have $\delta^{k-2}(\mathbb{I}_{[0,s]^{k-2}})=H_{k-2}(B_s)$ (with $H_k$ the $k$-th Hermite polynomial). Since $B$ as a Gaussian law, there is a real number $l>0$ and a set $\Omega_0\subset\Omega$ such that $\mathbb{P}(\Omega_0)>0$ and $\forall \omega\in \Omega_0, |\delta^{k-2}(\mathbb{I}_{[0,s]^{k-2}})(\omega)|>l.$
Then $\left|\frac{C_{s,t}(\omega)}{|t-s|^{2\kappa}}\right|\underset{s\rightarrow t}\longrightarrow +\infty$ for all fixed $s>0$ and $\omega\in \Omega_0$.} that $(u,P)$ does not \textit{a priori} belongs to $\mathcal{D}^{2\kappa}$ for some $\kappa>\frac{1}{2}$, and therefore we cannot directly apply the results of  Section \ref{sec:Controlled paths}.

To prove that $(u,P)\in\mathbb{C}_2$ we will proceed as follows.

The hypothesis of Proposition \ref{traceClass} are verified by $A$ and $C$. Indeed, $A_{s,\cdot},C_{s,\cdot}\in\mathbb{D}^{1,2}(|\mathcal H|)$ for all $s\in [0,T]$. Moreover, using the same arguments as in Step 1, one can show that $DA_{s,\cdot}$ and $DC_{s,\cdot}$ have almost continuous paths in $[0,T]^2$, implying in turn that
$$\int_0^T\int_0^T(|D_wA_{s,l}|+|D_wC_{s,l}|)|l-w|^{2H-2}dldw<\infty \quad\mbox{a.s.}$$
for all $s\in [0,T]$.

Formula \eqref{fracMal} allows to write
\begin{equation} \int_s^tC_{s,l}dB_l=\delta\left(C_{s,\cdot}\times\mathbb{I}_{[s,t]}(\cdot)\right)+c_H\int_s^t\int_0^tD_wC_{s,l}|l-w|^{2H-2}dwdl\label{decompC}\end{equation}
as well as
\begin{equation} \int_s^tA_{s,l}dB_l=\delta\left(A_{s,\cdot}\times\mathbb{I}_{[s,t]}(\cdot)\right)+c_H\int_s^t\int_0^tD_wA_{s,l}|l-w|^{2H-2}dwdl\label{decompC}.\end{equation}
For any $0\leq s\leq t\leq T$ and $0\leq x\leq y\leq T$, we can then write 
$L_{s,t}L_{x,y}=\sum_{i,j=1}^4R^{i,j}(s,t,x,y)$, with
\begin{eqnarray*}
R^{1,1}(s,t,x,y)&=&\delta\left(C_{s,\cdot}\mathbb{I}_{[s,t]}(\cdot)\right)\delta\left(C_{x,\cdot}\mathbb{I}_{[x,y]}(\cdot)\right)\\
R^{1,2}(s,t,x,y)
&=&c_H^2\int_s^tdl\int_0^tdw\int_x^ydr\int_0^ydz\\
&&\hskip3cm\times D_wC_{s,l}D_zC_{{x,r}}|w-l|^{2H-2}|z-r|^{2H-2}\\
R^{1,3}(s,t,x,y)&=&R^{1,4}(x,y,s,t)\\
&=&c_H\delta\left(C_{s,\cdot}\mathbb{I}_{[s,t]}(\cdot)\right)\int_x^y\int_0^yD_wC_{x,l}|l-w|^{2H-2}dwdl\\
R^{2,1}(s,t,x,y)&=&\delta\left(A_{s,\cdot}\mathbb{I}_{[s,t]}(\cdot)\right)\delta\left(A_{x,\cdot}\mathbb{I}_{[x,y]}(\cdot)\right)\\
R^{2,2}(s,t,x,y)&=&c_H^2\int_s^tdl\int_0^tdw\int_x^ydr\int_0^ydz\\
&&\hskip3cm\times D_wA_{s,l}D_zA_{{x,r}}|w-l|^{2H-2}|z-r|^{2H-2}\\
R^{2,3}(s,t,x,y)&=&R^{2,4}(x,y,s,t)\\
&=&c_H\delta\left(A_{s,\cdot}\mathbb{I}_{[s,t]}(\cdot)\right)\int_x^y\int_0^yD_wA_{x,l}|l-w|^{2H-2}dwdl\\
R^{3,1}(s,t,x,y)&=&R^{4,1}(x,y,s,t)\\
&=&\delta\left(A_{s,\cdot}\mathbb{I}_{[s,t]}(\cdot)\right)\delta\left(C_{x,\cdot}\mathbb{I}_{[x,y]}(\cdot)\right)\\
R^{3,2}(s,t,x,y)&=&R^{4,2}(x,y,s,t)\\
&=&\delta\left(A_{s,\cdot}\mathbb{I}_{[s,t]}(\cdot)\right)\int_x^y\int_0^yD_wC_{x,l}|l-w|^{2H-2}dwdl\\
R^{3,3}(s,t,x,y)&=&R^{4,3}(x,y,s,t)\\
&=&\delta\left(C_{s,\cdot}\mathbb{I}_{[s,t]}(\cdot)\right)\int_x^y\int_0^yD_wA_{x,l}|l-w|^{2H-2}dwdl\\
R^{3,4}(s,t,x,y)&=&R^{4,4}(x,y,s,t)\\
&=&c_H^2\int_s^tdl\int_0^tdw\int_x^ydr\int_0^ydz\\
&&\hskip3cm\times D_wC_{s,l}D_zA_{{x,r}}|w-l|^{2H-2}|z-r|^{2H-2}.
\end{eqnarray*}
We can easily check that 
\begin{equation*}
\mathbb{E}\left[R^{1,3}\right]=\mathbb{E}\left[R^{2,3}\right]=\mathbb{E}\left[R^{3,1}\right]=\mathbb{E}\left[R^{3,2}\right]=\mathbb{E}\left[R^{3,4}\right]=0.
\end{equation*}
Indeed, these expectations reduce to a sum of expectations of products of two multiple Wiener integrals of different orders, which are orthogonal in $L^2(\Omega)$ by Proposition \ref{multipleWiener}. 
More precisely, Lemma \ref{expectProduct} allows to show that all the  expectations in play vanish. For example,
$$\mathbb{E}[R^{1,3}]=c_H\int_x^y\int_0^y\mathbb{E}[\delta(C_{s,\cdot}\mathbb{I}_{[s,t]}(\cdot))D_wC_{x,l}]|l-w|^{2H-2}dwdl$$
which corresponds exactly to a term of the form (\ref{e1}).

We will now apply Proposition \ref{traceClass}, together with several inequalities, to show that all the remaining terms satisfy the condition (\ref{pseudo}), namely 
$$\mathbb{E}[R^{i,j}(s,t,x,y)]=o_{|t-s|+|x-y|\rightarrow 0}(f_2(s,t,x,y))$$ 
for all $(i,j)\in \{(1,1), (1,2), (2,1), (2,2), (3,3), (4,3)\}$ and uniformly in $[0,T]^2$.
(Starting from now, note that every time we write $o_{|t-s|+|x-y|\rightarrow 0}(f_2(s,t,x,y))$, it is implicitely assumed that it takes place uniformly in $s,t\in [0,T]$.)\\

Whatever the value of $(i,j)\in \{(1,1), (1,2), (2,1), (2,2), (3,3), (4,3)\}$, deriving a bound for $\EE[R^{i,j}]$ requires similar arguments. 
For this reason, in what follows we will fully develop the cases $(i,j)= (1,1)$, $(i,j)= (1,2)$ and $(i,j)=(2,1)$, then we will only explain the differences for the remaining cases.
\smallbreak For notational simplicity, we will also write $R^{i,j}$ instead of $R^{i,j}(s,t,x,y)$.\\

\noindent\textit{\underline{Step 3: Bound on $\EE[R^{1,2}]$}}.
First, we give an upper bound for $\mathbb{E}[(D_wC_{s,l})^2]$: for all $w\in [0,t]$ and $l\in [s,t]$,
\begin{eqnarray}\label{intermediate}
\mathbb{E}\left[\left(D_wC_{s,l}\right)^2\right]&=&k^2(k-1)^2(k-2)^2\,\mathbb{E}\left[\left(\delta^{k-3}\left(f_k(\cdot,w,s,s)\otimes_1\mathbb{I}_{[s,l]}\right)\right)^2\right]\notag\\ 
&=&k!k(k-1)(k-2)\left\|f_k(\cdot,w,s,s)\otimes_1\mathbb{I}_{[s,l]}\right\|^2_{\mathcal H^{\otimes k-3}}\notag\\
&\leq& k!k(k-1)(k-2)\|g_k\|^2_{\infty}\left\|\mathbb{I}_{[0,T]^{k-2}}\otimes_1\mathbb{I}_{[s,l]}\right\|^2_{|\mathcal H|^{\otimes k-3}},\notag
\end{eqnarray}
where, in the last inequality, we have used that $|h_1\otimes_1 h_2|\leq |h_1|\otimes_1|h_2|$ for all $h_1,h_2\in\mathbb{D}^{1,2}(|\mathcal H|)$.
Moreover, according to Lemma \ref{simpleInequality},
\begin{equation*}
\big|\mathbb{I}_{[0,T]^{k-2}}\otimes_1\mathbb{I}_{[s,l]}\big|=
\big|\mathbb{E}[B_T(B_s-B_l)]\mathbb{I}_{[0,T]^{k-3}}\big|
\leq K|s-l|\mathbb{I}_{[0,T]^{k-3}}.\end{equation*} 
Plugging this identity into \eqref{intermediate} leads to
\begin{equation*}\mathbb{E}\left[(D_wC_{s,l})^2\right]\leq K|s-l|^2.
\end{equation*}
As a result, and using the H\"older inequality, we have, for all $s\leq t$ and $x\leq y$,
\begin{eqnarray*}
&&\left|\mathbb{E}\left[
\int_s^t\int_0^tD_wC_{s,l}|w-l|^{2H-2}dwdl\int_x^y\int_0^yD_wC_{x,l}|w-l|^{2H-2}dwdl
\right]\right|\\
&\leq&
\int_s^t\int_0^t\int_x^y\int_0^y\EE\left[\left(D_wC_{s,l}\right)^2\right]^\frac{1}{2}\EE\left[\left(D_zC_{x,r}\right)^2\right]^\frac{1}{2}\\&&\times|w-l|^{2H-2}|z-r|^{2H-2}dzdrdwdl\\
&\leq& K|t-s|^2|x-y|^2=o_{|t-s|+|x-y|\rightarrow 0} f_2(s,t,x,y)
\end{eqnarray*} 
where, in the last identity, we made use of the following two facts: on one hand
$|t-s||x-y|\leq r_H(s,t,x,y)$ according to Lemma \ref{simpleInequality}; on the other hand, and since $H\leq \frac34$,
$$ |t-s||x-y|= o_{|t-s|+|x-y|\rightarrow0}\big(|t-s|^{2H-1}|x-y|^{2H-1}\kappa_H(|x-y|)\kappa_H(|t-s|)\big).$$

\noindent\textit{\underline{Step 4: Bound on $\EE[R^{1,1}]$}}.
This term can be handled similarly, with the help of Proposition \ref{boundsSkorokhod}:
\begin{eqnarray*}
&&\left|\EE\left[\delta(C_{s,\cdot}\mathbb{I}_{[s,t]}(\cdot))\delta(C_{x,\cdot}\mathbb{I}_{[x,y]}(\cdot))\right]\right|\\
&\leq&\int_{[s,t]\times[x,y]}\mathbb{I}_{[s,t]}(l)\mathbb{I}_{[x,y]}(r)|\EE[C_{s,l}C_{x,r}]||l-r|^{2H-2}drdl\\
&&+\int_{[s,t]\times[x,y]}\int_{[0,s]\times[0,t]}|\EE[D_wC_{s,l}D_zC_{x,r}]||l-r|^{2H-2}|z-w|^{2H-2}dzdwdrdl\\
&\leq& K\|g_k\|^2_{\infty}|t-s||x-y|r_H(s,t,x,y),
\end{eqnarray*}
where $\EE[C_{s,l}C_{x,r}]$ and $\EE[D_uC_{s,l}D_vC_{x,r}]$ are computed by means of Proposition \ref{multipleWiener}. Again, $|t-s||x-y|r_H(s,t,x,y)=o_{|t-s|+|x-y|\rightarrow 0}(f_2(s,t,x,y))$.\\

\noindent\textit{\underline{Step 5: Bound on $\EE[R^{2,1}]$}}.
Using Proposition \ref{boundsSkorokhod}, we can write

\begin{eqnarray*}
&&|\mathbb{E}[R^{2,1}]|=\left|\EE\left[\delta(A_{s,\cdot}\mathbb{I}_{[s,t]}(\cdot))\delta(A_{x,\cdot}\mathbb{I}_{[x,y]}(\cdot))\right]\right|\\
&\leq&\int_{[s,t]\times[x,y]}dldr\mathbb{I}_{[s,t]}(l)\mathbb{I}_{[x,y]}(r)\,\big|\EE[A_{s,l}A_{x,r}]\big|\,|l-r|^{2H-2}\\
&+&\int_{[s,t]\times[x,y]}dldr\int_{[0,t]\times[0,y]}dwdz\,\big|\EE[D_wA_{s,l}D_zA_{x,r}]\big|\,|l-r|^{2H-2}|z-w|^{2H-2}.
\end{eqnarray*}
Let us define the following function:
\begin{eqnarray*}
&&h^s_{k}(x_1,\ldots,x_k,l)\\
&:=&\sum_{i=1}^{k}\mathbb{I}_{[0,s]^{k-1}}(x_1,\ldots,x_{i-1},x_{i+1},\ldots,x_k)\mathbb{I}_{[s,l]}(x_i)\\
&&\times\left(g_k(x_1,\ldots,x_k)-g_{k}(x_1,\ldots,x_{i-1},x_{i+1},\ldots,x_{k},s)\right)\\
&&+g_k(x_1,\ldots,x_k)\sum_{i=1}^k\mathbb{I}_{[s,l]}(x_i)\mathbb{I}_{[0,l]^{k-1}\setminus[0,s]^{k-1}}(x_1,\ldots,x_{i-1},x_{i+1},\ldots,x_n).
\end{eqnarray*}
Since $s\leq l$, we have:
\begin{eqnarray*}
&&f_k(x_1,\ldots,x_k,l)-f_k(x_1,\ldots,x_k,s)\\&=&g_k(x_1,\ldots,x_k)\sum_{i=1}^k\mathbb{I}_{[s,l]}(x_i)\mathbb{I}_{[0,s]^{k-1}}(x_1,\ldots,x_{i-1},x_{i+1},\ldots,x_k)\\
&&+g_k(x_1,\ldots,x_k)\sum_{i=1}^k\mathbb{I}_{[s,l]}(x_i)\mathbb{I}_{[0,l]^{k-1}\setminus[0,s]^{k-1}}(x_1,\ldots,x_{i-1},x_{i+1},\ldots,x_k)
\end{eqnarray*}
and
\begin{eqnarray*}
&&k\widetilde{f_k(\cdot,s,s)\otimes\mathbb{I}_{[s,l]}}(x_1,\ldots,x_k)\\
&=&\sum_{i=1}^k\mathbb{I}_{[s,l]}(x_i)\mathbb{I}_{[0,s]^{k-1}}g_k(x_1,\ldots,x_{i-1},x_{i+1},\ldots,x_{k},s).
\end{eqnarray*}
We obtain, for all $x_1,\ldots,x_k\in [0,T]$, that
\begin{eqnarray*}
&&f_k(x_1,\ldots,x_k,l)-f_k(x_1,\ldots,x_k,s)-k\widetilde{f_k(\cdot,s,s)\otimes\mathbb{I}_{[s,l]}}(x_1,\ldots,x_k)\\
&=& h^s_{k}(x_1,\ldots,x_k,l).
\end{eqnarray*}
 Then, $A_{s,l}=\delta^k(h^s_k(\cdot,l)),\,A_{x,r}=\delta^k(h^x_k(\cdot,r))$ and, by Proposition \ref{multipleWiener} (isometry),
\begin{eqnarray*}
&&\int_{[s,t]\times[x,y]}\mathbb{I}_{[s,t]}(l)\mathbb{I}_{[x,y]}(r)\big|\EE[A_{s,l}A_{x,r}]\big|\,|l-r|^{2H-2}drdl\\
&\leq&k!c_H^{k+1}\int_s^tdl\int_x^ydr\int_{[0,t]^{k}\times[0,y]^k}\,|h^s_{k}(x_1,\ldots,x_{k},l)|\,\left|h^x_{k}(y_1,\ldots,y_{k},r)\right| \\
&&\times\prod_{i=1}^{k}|x_i-y_i|^{2H-2}|l-r|^{2H-2}dx_1\ldots dx_{k}dy_1\ldots dy_{k}.\\
\end{eqnarray*}
On the other hand, observe the following facts:
\begin{itemize}
\item $h^s_k(x_1,\ldots, x_k)=0$ if $(x_1,\ldots, x_k)\in [0,s]^k$;
\item if there is a unique index $i$ such that $x_i\in [s,l],$ then 
\begin{eqnarray*}
&&|h^s_k(x_1,\ldots, x_k)|=|g_k(x_1,\ldots,x_k)-g_k(x_1,\ldots, x_{i-1},x_{i+1},\ldots,x_k,s)|\\
&\leq& \|g_k\|_{\beta}|s-l|^{\beta};
\end{eqnarray*}
\item if there is more than one index $i$ such that $x_i\in [s,l]$, then
$$|h^s_{k}(x_1,\ldots,x_k,l)|\leq \|g_k\|_{\infty}\mathbb{I}_{[0,l]}(x_1,\ldots,x_k)\sum_{i\neq j=1}^k\mathbb{I}_{[s,l]^2}(x_i,x_j).$$
\end{itemize}
As a result,
\begin{eqnarray*}
&&|h^s_k(x_1,\ldots,x_k,l)|\\
&\leq&\sum_{i=1}^{k-1}\mathbb{I}_{[0,s]^k}(x_1,\ldots,x_{i-1},x_{i+1},\ldots,x_k)\mathbb{I}_{[s,l]}(x_i)|x_i-s|^{\beta}\|g_k\|_{\beta}\\
&&+\|g_k\|_{\infty}\sum_{i\neq j=1}^kg_k(x_1\ldots,x_k)\mathbb{I}_{[s,l]^2}(x_i,x_j)\mathbb{I}_{[0,l]^k}(x_1\ldots,x_k).
\end{eqnarray*}
We then have
\begin{eqnarray*}
&&c_H^{k+1}\int_s^t\int_x^y\int_{[0,l]^{k}\times[0,v]^k}|h^s_{k}(x_1,\ldots,x_{k},l)|\left|h^x_{k}(y_1,\ldots,y_{k},r)\right| \\
&&\times\prod_{i=1}^{k}|x_i-y_i|^{2H-2}|l-r|^{2H-2}dx_1\ldots dx_{k}dy_1\ldots dy_{k}drdl\\
&\leq& (A+B+C+D)r_H(s,t,x,y),
\end{eqnarray*}
with
\begin{eqnarray*}
&A=&c_H^k\|g_k\|^2_{\beta}|t-s|^\beta|x-y|^\beta\sum_{i,j=1}^k\int_{[0,t]^{k}\times[0,y]^k}\mathbb{I}_{[s,t]}(x_i)\mathbb{I}_{[x,y]}(y_j) \\
&&\times\prod_{m=1}^{k}|x_m-y_m|^{2H-2}dx_1\ldots dx_{k}dy_1\ldots dy_{k}\\
&B=&c_H^k\|g_k\|^2_{\infty}\sum_{i_i\neq i_2,j_1\neq j_2=1}^k\int_{[0,t]^{k}\times[0,y]^k}\mathbb{I}_{[s,t]^2}(x_{i_1},x_{i_2})\mathbb{I}_{[x,y]^2}(y_{j_1},y_{j_2}) \\
&&\times\prod_{m=1}^{k}|x_m-y_m|^{2H-2}dx_1\ldots dx_{k}dy_1\ldots dy_{k}
\end{eqnarray*}
\begin{eqnarray*}
&C=&c_H^k|x-y|^{\beta}\|g_k\|_{\beta}\|g_k\|_{\infty}\sum_{i_i\neq i_2,j=1}^k\int_{[0,t]^{k}\times[0,y]^k}\mathbb{I}_{[s,t]^2}(x_{i_1},x_{i_2})\mathbb{I}_{[x,y]}(y_j) \\
&&\times\prod_{m=1}^{k}|x_m-y_m|^{2H-2}dx_1\ldots dx_{k}dy_1\ldots dy_{k}\\
&D=&c_H^k|t-s|^{\beta}\|g_k\|_{\beta}\|g_k\|_{\infty}\sum_{i,j_1\neq j_2=1}^k\int_{[0,t]^{k}\times[0,y]^k}\mathbb{I}_{[s,t]}(x_i)\mathbb{I}_{[x,y]^2}(y_{j_1},y_{j_2}) \\
&&\times\prod_{m=1}^{k}|x_m-y_m|^{2H-2}dx_1\ldots dx_{k}dy_1\ldots dy_{k}.\\
\end{eqnarray*}
We only write down the details for the upper bound  of $A$, since the technique is similar for the three other terms.

Two cases should then be analyzed to handle the integral $A$:
\begin{itemize}
\item $i\neq j$:
\begin{eqnarray*}
&&c_H^k\int_{[0,t]^{k}\times[0,y]^k}\mathbb{I}_{[s,t]}(x_{i})\mathbb{I}_{[x,y]}(y_{j})\times\prod_{m=1}^{k}|x_m-y_m|^{2H-2}dx_1\ldots dx_{k}dy_1\ldots dy_{k}\\
&=& \mathbb{E}[B_tB_y]^{k-2}\mathbb{E}[B_y(B_t-B_s)]\mathbb{E}[B_t(B_y-B_x)]\leq K^2_TT^{2H(k-2)}|t-s||x-y|,
\end{eqnarray*}
where the last inequality follows from Lemma \ref{simpleInequality}.\\
\item $i= j$:
\begin{eqnarray*}
&&c_H^k\int_{[0,t]^{k}\times[0,y]^k}\mathbb{I}_{[s,t]}(x_{i})\mathbb{I}_{[x,y]}(y_{j})\times\prod_{m=1}^{k}|x_m-y_m|^{2H-2}dx_1\ldots dx_{k}dy_1\ldots dy_{k}\\
&=& \mathbb{E}[B_tB_y]^{k-1}\mathbb{E}[(B_x-B_y)(B_t-B_s)]\\
&\leq& T^{2H(k-1)}r_H(s,t,x,y)\leq T^{2H(k-1)}|t-s|^H|x-y|^H,
\end{eqnarray*}
where the last inequality comes from Lemma \ref{simpleInequality}. We then have 
\begin{equation*}
A\leq K(|t-s||x-y|+|t-s|^H|x-y|^H)|t-s|^\beta|x-y|^\beta.
\end{equation*}
\end{itemize}

Similar arguments for handling the integrals $B,C,D$ lead to
\begin{eqnarray*}
B&\leq&K(|t-s|^{2H}|x-y|^{2H}+|t-s|^H|x-y|^H|t-s||x-y|+|t-s|^2|x-y|^2)\\
C&\leq&K|x-y|^\beta(|x-y|^H|t-s|^{1+H}+|x-y||t-s|^2)\\
D&\leq&K|t-s|^\beta(|t-s|^H|x-y|^{1+H}+|t-s||x-y|^2).
\end{eqnarray*}

Since $\beta, H>\frac{1}{2}$, we have
\begin{eqnarray*}
&&\left|\int_{[s,t]\times[x,y]}\mathbb{I}_{[s,t]}(l)\mathbb{I}_{[x,y]}(r)|\EE[A_{s,l}A_{x,r}]||l-r|^{2H-2}drdl\right|\\
&\leq& r_H(s,t,x,y)(A+B+C+D)=o_{|t-s|+|x-y|\rightarrow 0}(f_2(s,t,x,y)).
\end{eqnarray*}
\bigskip We have $D_wA_{s,l}=\delta^{k-1}(h_k^s(x_1,\ldots,x_{k-1},u,l))$. Similar computations allow to treat the trace term:
\begin{eqnarray*}&&\int_{[s,t]\times[x,y]}\int_{[0,s]\times[0,t]}|\EE[D_wA_{s,l}D_zA_{x,r}]||l-r|^{2H-2}|z-w|^{2H-2}dzdwdrdl\\&=&o_{|t-s|+|x-y|\rightarrow 0}(f_2(s,t,x,y)).\end{eqnarray*}

Putting all these facts together, we obtain $$\EE[R^{2,1}]=o_{|t-s|+|x-y|\rightarrow 0}(f_2(s,t,x,y)).$$

\bigskip

\noindent\textit{\underline{Step 6: Bound on $\EE[R^{2,2}+R^{3,3}+R^{4,3}]$}}.
We use similar arguments here as in Step 5: we can obtain trough easy but tedious computations, and distinguishing again several cases,
$$\mathbb{E}\left[R^{2,2}+R^{3,3}+R^{4,3}\right]=o_{|t-s|,|x-y|\rightarrow 0}f_2(s,t,x,y).$$

\bigskip

\noindent\underline{\textit{Step 7: Conclusion}}.
We have shown that
\begin{equation*}\mathbb{E}\left[L_{s,t}L_{x,y}\right]=o_{|t-s|,|x-y|\rightarrow0}(f_2(s,t,x,y))
\end{equation*}
implying that $(u,P)\in \mathbb{C}_2$.
\qed

\subsection{Examples in the Brownian motion case}\label{sec:brown}

Since this section only concerns the standard Brownian motion case, in the following  $H=\frac12$.
To illustrate the novelty of our approach compared to that followed by Rootz\'en in \cite{rootzen1980limit}, we develop specific examples which cannot be obtained from \cite{rootzen1980limit}.

In Proposition \ref{integralProcesses}, we considered fractional semimartingales of the form (\ref{integralProcess}). Here, we take advantage of the standard Brownian framework, to consider processes of the form (\ref{18}). Note that the integrand $V^{i, j}_{s,t} $ is allowed to depend on $t$ in  (\ref{18}), making useless to consider a drift term as in (\ref{integralProcess}).

Let $\big((u^i_t)_{t\in[0,T]}\big)_{1\leq i\leq m}$ be a collection of square integrable and progressively measurable processes, i.e. $\mathbb{E}\big[(u^i_t)^2\big]<\infty$ for all $i$ and $t$.
According to the representation theorem for square integrable random variables, for all $i$ and $t$ there exists progressively measurable processes
$\big((V^{i,j}_{s,t})_{0\leq s\leq t}\big)_{1\leq j\leq d}$ such that, for all $i$ and $t$: 
\begin{equation}\label{18}
u^i_t=\mathbb{E}[u^i_t]+\sum_{j=1}^d\int_0^tV^{i,j}_{s,t}dB^j_s\quad\mbox{a.s.},
\end{equation}
and $\mathbb{E}\big[\int_0^t (V^{i,j}_{s,t})^2ds\big]<\infty$.
We assume moreover:
\begin{itemize}
\item[$(\mathfrak H_1)$] $(V_{s,t}^{i,j})_{0\leq s\leq t\leq T}$ is measurable for all $i$ and $j$, and $(i)$ $(s,t)\mapsto V_{s,t}$ has a progressively measurable version, $(ii)$
$\mathbb{E}[|V^{i,j}_{s,s}-V^{i,j}_{s,t}|^2]+\mathbb{E}[|V^{i,j}_{s,s}-V^{i,j}_{t,t}|^2]\underset{s\rightarrow t-}\longrightarrow 0$  for all $i,j$ uniformly in $s\leq t\in [0,T]$ and $(iii)$ $(V^{i,j}_{s,s})_{s\in[0,T]}$ is piecewise continuous.
\item[$(\mathfrak H_2)$] For all $i,j$, the family $\left(|V^{i,j}_{s,t}|\right)_{s,t\in[0,T]}$ is bounded by a square integrable random variable $S$ such that $\mathbb{E}[S^{2+\gamma}]<\infty$ for some $\gamma>0$.
\item[$(\mathfrak H_3)$] One has, for all $0\leq s\leq t\leq T$ and all $i\leq m$ and $j\leq d$
\begin{equation*}  
\mathbb{E}\left[\int_0^s(V^{i,j}_{l,s}-V^{i,j}_{l,t})^{2}dl\right]+\big(\mathbb{E}[u^i_{s}-u^i_{t}]\big)^2\leq |s-t|\mu(s,t),
\end{equation*} 
where $\mu$ is a bounded function which is continuous on $[0,T]^2$ and such that $\mu(s,s)=0$ for all $s\in [0,T]$.
\end{itemize}

As an application of Theorem \ref{main2} (with $P_s=V_{s,s}$), we can state the following proposition. 

\begin{proposition}\label{brownianCriterion} \textit{
Assume $(\mathfrak H_1)-(\mathfrak H_3)$ and recall that $H=\frac12$. Then,  stably in $\mathcal C_{\mathbb{R}^{m\times d}}([0,T])$,
\begin{equation}
\{\sqrt{n}M^{n,(i,j)}_\cdot\}_{1\leq i\leq m,1\leq j\leq d} \overset{}{\underset{n\rightarrow\infty}{\longrightarrow}}\left\{
\sum_{k=1}^d\int_0^\cdot V^{i,k}_{s,s}dW^{k,j}(s)
\right\}_{1\leq i\leq m,1\leq j\leq d},\notag
\end{equation}
where $W$ is the independent matrix-valued Brownian motion of Section \ref{rosenblatt-and-brownian}.
}
\end{proposition}
\begin{proof}
To simplify, without loss of generality we assume that $m=1$.  
We then write $P^j=P^{1,j}$, $V^j=V^{1,j}$ and $L^j=L^{1,j}$ for all $1\leq j\leq d$. 

Given $(\mathfrak H_1)$, (iii) and $(\mathfrak H_2)$, we have that $s\mapsto P_s$ is piecewise continuous over $[0,T]$,
with $\mathbb{E}\left[\|P_.\|_{\infty}^{2+\gamma}\right]<+\infty$. Thus, it remains to check that
$(u,P)\in\mathbb{C}_2$. Since we are dealing with the standard Brownian case and since $s\leq t$ and $x\leq y$, 
we note that $r_H(s,t,x,y)=((t\wedge y)-(s\vee x))_{+}$. Thanks to the independence of increments, we are then left to check that $\forall j\in\{1\ldots,d\},$
$$\mathbb{E}[L^j_{s,t}L^j_{x,y}]=\sqrt{|t-s||x-y|}\,\times\,o_{|t-s|+|x-y|\rightarrow0}\big(((t\wedge y)-(s\vee x))_{+}\big).$$
We have, for all $1\leq j\leq d$ and with $\mathbb{B}^{i,j}_{s,t}=\int_s^t(B^i_l-B^i_{s})dB^j_l$,
\begin{eqnarray*}
L^j_{s,t}&=&\int_s^tu_ldB_l^j-u_{s}(B^j_{t}-B^j_{s})-\sum_{i=1}^dP^i_s\mathbb{B}^{i,j}_{s,t}\\
& =&\int_s^t\left(\mathbb{E}\left[u_l\right]-\mathbb{E}\left[u_s\right]\right)dB^j_l\\
&+&\int_s^t\left(\sum_{i=1}^d\int_0^l\left(\left(V^i_{x,l}-V^i_{x,s}\right)\mathbb{I}_{[0,s]}(x)+\left(V^i_{x,l}-V^i_{s,s}\right)\mathbb{I}_{[s,l]}(x)\right)dB^i_{x}\right)dB^j_l\\
&=:&L^{1,j}_{s,t}+L^{2,j}_{s,t}.
\end{eqnarray*} 

Let $s\leq t$ and $x\leq y$ be such that $s\vee x\leq t\wedge y$. The hypothesis $(\mathfrak H_3)$ allows us to write
\begin{eqnarray}\label{L1}
\mathbb{E}[L^{1,j}_{s,t}L^{1,j}_{x,y}]&=&\int_{s\vee x}^{t\wedge y}\mathbb{E}[u_l-u_s]\mathbb{E}[u_l-u_x]dl\nonumber\\
&\leq &\sqrt{\int_{s\vee x}^{t\wedge y}\left(\mathbb{E}[u_l-u_s]\right)^2dl\int_{s\vee x}^{t\wedge y}\left(\mathbb{E}[u_l-u_x]\right)^2dl}\nonumber\\
&\leq& ((t\wedge y)-(s\vee x))_{+}\sqrt{|t-s||x-y|}\sqrt{\sup_{l\in [{s\vee x},{t\wedge y}]}\mu(s,l)\mu(x,l)}.
\end{eqnarray}
We also have:
\begin{eqnarray*}
\mathbb{E}[L^{2,j}_{s,t}L^{2,j}_{x,y}]&=&\int_{s\vee x}^{t\wedge y}dl\sum_{i=1}^d\mathbb{E}\left[\int_0^{l}((V^i_{r,l}-V^i_{r,s})\mathbb{I}_{[0,s]}(r)+(V^i_{r,l}-V^i_{s,s})\mathbb{I}_{[s,l]}(r))dB^i_r\right.\\
&&\left.\quad\times\int_0^{y}((V^i_{r,l}-V^i_{r,x})\mathbb{I}_{[0,x]}(r)+(V^i_{r,l}-V^i_{x,x})\mathbb{I}_{[s,l]}(r))dB^i_r\right].
\end{eqnarray*}
Moreover, thanks to the isometry property, the Cauchy-Schwarz inequality and the assumption $(\mathfrak H_3)$, we can write, for all $i\leq d$,
\begin{eqnarray}
&&\mathbb{E}\left[\int_0^{t}((V^i_{r,l}-V^i_{r,s})\mathbb{I}_{[0,s]}(r)+(V^i_{r,l}-V^i_{s,s})\mathbb{I}_{[s,l]}(r))dB^i_r\right.\nonumber\\&&\left.\quad\times\int_0^{y}((V^i_{r,l}-V^i_{r,x})\mathbb{I}_{[0,s]}(r)+(V^i_{r,l}-V^i_{x,x})\mathbb{I}_{[s,l]}(r))dB^i_r\right]\nonumber\\
&\leq&\sqrt{|t-s|\mu(s,t)+\sup_{r\in [s,l]}\mathbb{E}[V^i_{r,l}-V^i_{s,s}]^2}\sqrt{|x-y|\mu(x,y)+\sup_{r\in [x,l]}\mathbb{E}[V^i_{r,l}-V^i_{x,x}]^2}.\nonumber\\
\label{L2}
\end{eqnarray}
Using the Cauchy-Schwarz inequality and then \eqref{L1} and \eqref{L2}, we finally obtain
\begin{eqnarray*}
&&\mathbb{E}[L^{1,j}_{[s,t]}L^{2,j}_{[x,y]}+L^{2,j}_{[s,t]}L^{1,j}_{[x,y]}]\\
&\leq& ((t\wedge y)-(s\vee x))_{+}\sqrt{|t-s|(\mu(s,t)+\sup_{l\in [s,t],r\in [s,l]}\mathbb{E}[V^i_{r,l}-V^i_{s,s}]^2)}\\
&&\hskip7cm \times\sqrt{|x-y|\sup_{l\in [x,y]}\mu(x,l)}\\
&&+((t\wedge y)-(s\vee x))_{+}\sqrt{|t-s|(\mu(s,t)+\sup_{l\in [s,t],r\in [s,l]}\mathbb{E}[V^i_{r,l}-V^i_{s,s}]^2)}\\
&&\hskip7cm \times
\sqrt{|x-y|\sup_{l\in [x,y]}\mu(x,l)}.
\end{eqnarray*}
Thanks to $(\mathfrak H_3)$ we have that the function $(s,t)\rightarrow\sup_{x\in [s,t]}\mu(s,t)$ is uniformly continuous on $[0,T]^2$ and since $\mu(t,t)=0$ for all $t$, \begin{equation*}\sup_{s,t\in [0,T], |s-t|\leq \delta}\sup_{x\in [s,t]}\mu(s,t)\underset{\delta\rightarrow 0}\longrightarrow 0.\end{equation*}
On the other hand, we have thanks to $(\mathfrak H_1)$,\begin{equation*} \sup_{s,t\in [0,T], s\leq t, |s-t|\leq \delta}\sup_{x\in [s,t]}\mathbb{E}[(V^i_{x,l}-V^i_{s,s})^2]\underset{\delta\rightarrow 0}\longrightarrow 0.\end{equation*}
Finally, $(u,P)\in \mathbb{C}_2$.
\end{proof}

We obtain a result analogous to Proposition \ref{integralProcesses} for semimartingale processes but with weaker hypotheses on the volatility $a$ and the drift $b$.

\begin{corollary}\textit{ 
Assume $m=1$, and consider $$u_t=u_0+\sum_{j=1}^d \int_0^t a^j_s dB^j_s+\int_0^tb_sds.$$
Assume that $a^j$ is progressively measurable and  piecewise continuous for any $j$, that $b$ is progressively measurable, that $g(s,t)=\sum_{k=1}^d\mathbb{E}\left[(a^k_s-a^k_t)^2\right]$ is continuous as a function of two variables, that $u_0$ is independent of $B$ and that for some $\gamma>0,$ $$\mathbb{E}\left[\max_{1\leq j\leq d}\|a^j\|_{\infty}^{2+\gamma}\right]+\mathbb{E}\left[\|b\|_{\infty}^{2+\gamma}\right]<+\infty.$$ Then, with
$M^{n,1,j}$ defined by (\ref{mnij}),
we have, stably in $\mathcal C_{\mathbb{R}^d}([0,T])$
\begin{eqnarray*}
\left\{\sqrt{n}M^{n,1,j}_\cdot\right\}_{1\leq j\leq d}\overset{}{\underset{n\rightarrow\infty}{\longrightarrow}}\left\{\sum_{i=0}^d\int_0^\cdot a^i_sdW^{i,j}_s\right\}_{1\leq j\leq d}
\end{eqnarray*}
where $W$ is the independent matrix-valued Brownian motion of Section \ref{rosenblatt-and-brownian}, see (\ref{wh}).
}
\end{corollary}
\begin{proof}
We have that the function $f:t\rightarrow\int_0^{t}b_sds$ is a.s.\!\! continuous and satisfies $\EE[\|f\|^{2+\gamma}_{\infty}]<\infty$.
Using Jensen inequality and the isometry property, we easily see that 
\begin{eqnarray*}
\left|\mathbb{E}\left[\int_s^t\left(\int_{s}^lb_udu\right)dB^j_l\int_x^y\left(\int_{x}^lb_udu\right)dB^j_l\right]\right|\\\leq |x-y||t-s|(t\wedge y-s\vee x)_+\sup_{l\in[0,T]}\mathbb{E}[b_l^2],
\end{eqnarray*} 
that is, $(\int_0^\cdot b_sds,0)\in\mathbb{C}_2$. Then, Theorem \ref{main2} applies, and 
\begin{equation*}\forall j\in\{1\ldots,d\},\int_0^t dl \int_0^lb_s\,dB^i_s-\sum_{k=1}^{nt_n}\int_0^\frac{k}{n}b_ldl(B^j_{\frac{k+1}{n}\wedge t}-B^j_\frac{k}{n})\overset{\mathcal C([0,T])}{\underset{n\rightarrow\infty}\longrightarrow}0.
\end{equation*}
Moreover, we can apply Proposition \ref{brownianCriterion} to $\int_0^\cdot a_sdB_s$ with $V^{1,j}_{s,t}=a^j_s\mathbb{I}_{[0,t]}(s)$ (all its assumptions are satisfied). Slutsky's lemma allows finally to conclude.
\end{proof}

Unlike the case $H>\frac{1}{2}$, here we can allow the volatility process $a$ to be discontinuous. An illustration of this fact is given by choosing $d=1$, $(T_i)_{i\geq 1}$ a sequence of increasing stopping times such that $T_i\underset{i\rightarrow\infty}\longrightarrow\infty$ a.s, a sequence $(x^i)_{i\geq 1}\in \mathbb{R}^{{N}^*}$ of progressively measurable processes on $[0,T]$ such that $\sum_i\|x^i\|_{\infty}^2<\infty$, and
\begin{equation*}
u_t=\sum_{i\geq 1}\int_0^{t\wedge T_i}x^i_sdB_{s}.
\end{equation*}
We then have, stably in $\mathcal C_{\mathbb{R}}([0,T])$,
\begin{equation*}
\sqrt{n}M_\cdot^{n}\overset{}\longrightarrow\frac{1}{\sqrt{2}}\int_0^\cdot\sum_ix_s^i\mathbb{I}_{[0,T_i]}(s)dW_s.
\end{equation*}

\subsection{Irregular processes}\label{sec3.4}
In this section,  $H\in(\frac{1}{2},\frac{2}{3})$.
We state a first order convergence for a general class of processes possessing mild regularity properties.

Although the process $u$ considered in Proposition \ref{prp} is of the form $u_s=F(B_s)$, 
 the fact that $F$ is supposed to be convex allows potential discontinuities for $F'$, and it becomes hopeless to expect a second order result as obtained in Corollary \ref{functionB}
in a seemingly similar framework.

\begin{proposition} \label{prp}
 Let $u_s=F(B_s)$, $s\in [0,T]$, with $F$ a real convex function such that, for some $K>0$ and $\gamma\in(0,2)$, 
\begin{equation*}
|F(x)|+|F'(x)|+\int_{-|x|}^{|x|}(|a|+1)dF''(a)\leq Ke^{|x|^{\gamma}},\quad x\in\R,
\end{equation*}
where $F'$ is the right derivative of $F$ and $F''$ denotes its second derivative in the distributional sense (a simple `non-smooth' example is given by $x\rightarrow|x|$). Then,  for all $t\in [0,T]$,
\begin{eqnarray*} 
M^n_t&:=&n^{2H-1}\left(\int_0^t F(B_s)dB_s-\sum_{k=0}^{\lfloor nt \rfloor }F(B_\frac{k}{n})(B_{\frac{k+1}{n}\wedge t}-B_\frac{k}{n})\right)\\
&&
\overset{L^2(\Omega)}{\underset{n\rightarrow\infty}\longrightarrow}\frac{1}{2}\int_0^t F'(B_s)ds.
\end{eqnarray*}
\end{proposition}

\begin{proof} The proof is divided into two steps: in the first one, we will first show that $u_s=F(B_s)$ belongs to $\mathbb{D}^{1,2}(|\mathcal H|)$ and give a suitable expression for its Malliavin derivative. 
This is then in Step 2 that we will show the $L^2(\Omega)$-convergence of $M^n$, with the help of Proposition \ref{traceClass} and of Lemma \ref{boundOnSign}.

\medskip\textit{\underline{Step 1: $u$ belongs to $\mathbb{D}^{1,2}(|\mathcal H|)$}.}
Consider the truncated function
$$F^n:x\rightarrow F(x)\mathbb{I}_{|x|\leq n}+F(n)\mathbb{I}_{x>n}+F(-n)\mathbb{I}_{x<-n}.$$
Every convex function is locally Lipschitz continuous so the previous sequence is Lipschitz continuous. Then, by a slight extention of \cite[Proposition 2.3.8]{nourdin2012normal}, we know that the process $u^n_s=F^n(B_s)$ belongs to $\mathbb{D}^{1,2}(|\mathfrak H|)$, and $D_su^n_t=(F^n)'(B_t)\mathbb{I}_{s\leq t}$. Moreover, $F^n\to F$ and $(F^n)'\to F'$ pointwise as $n\to\infty$, and the growth condition on $F$ and $F'$ ensures that, for all $p>2$, the sequences $F^n(B_s)$ and $(F^n)'(B_s)$ are bounded in $L^p(\Omega,|\mathcal H|)$ and $L^p(\Omega,|\mathcal H|\times|\mathcal H|)$ respectively. Then, these sequences are both uniformly integrable in $L^2$, and the bounded convergence theorem ensures that, as $n\to\infty$,
\begin{eqnarray*}
&&u^n\to u\quad\mbox{in $L^p(\Omega,|\mathcal H|)$}\\
&&\mathbb{I}_{\{\cdot\leq \cdot\cdot\}}\,D_\cdot u^n_{\cdot\cdot}\to \mathbb{I}_{\{\cdot\leq \cdot\cdot\}}\,F'(B_{\cdot\cdot})\quad\mbox{in $L^p(\Omega,|\mathcal H|\times|\mathcal H|)$.}
\end{eqnarray*}
Then, $u\in\mathbb{D}^{1,2}(|\mathcal H|)$, with $D_su_t=\mathbb{I}_{s\leq t}F'(B_t)$.
Since $F'$ is locally bounded, the process $u$ verifies the assumptions of Proposition \ref{traceClass}.

\medskip\textit{\underline{Step 2: $L^2$ convergence}.}
By e.g. \cite[page 224]{revuz2013continuous}, we know that, for all $k\in\mathbb{N}^*$, there exist $\alpha_k,\beta_k\in\mathbb{R}$ such that
\begin{equation*}F(x)=\alpha_k+\beta_kx+\frac{1}{2}\int_{-k}^{k}|x-a|dF''(a),\quad x\in [-k,k].
\end{equation*}
Then, for all $x\in\mathbb{R}$,
\begin{equation*}F(x)=\sum_{k=0}^{+\infty}\left(\alpha_{k+1}+\beta_{k+1}x+\frac{1}{2}\int_{-k-1}^{k+1}|x-a|dF''(a)\right)\mathbb{I}_{[-k-1,-k)\cup[k,k+1)]}(x).
\end{equation*}
Since $F$ is convex, $dF''$ can be identified with a Radon measure, which is $\sigma$-finite. This allows us to interchange the integrals and derivatives. Since $D_{\cdot}u_{\cdot\cdot}=F'(B_{\cdot\cdot})\mathbb{I}_{\cdot\leq\cdot\cdot}$ we can then rewrite $Du$ as:
\begin{equation}\label{deriv}D_su_t=\mathbb{I}_{t\geq s}\sum_{k=0}^{+\infty}\left(\beta_{k+1}+\frac{1}{2}\int_{-k-1}^{k+1}{\rm sign}(B_t-a)dF''(a)\right)\mathbb{I}_{[-k-1,-k)\cup[k,k+1)]}(B_t),
\end{equation} 
where $\rm sign$ is the left derivative of $x\rightarrow|x|$.

Let $0\leq t\leq T$. We have, thanks to Proposition \ref{traceClass} and recalling that $s_n=\frac1n\lfloor ns\rfloor$,
\begin{eqnarray*}&&M_{t}^n-\frac{1}{2}\int_{0}^{t}F'(B_s)ds\\
&=&n^{2H-1}\Big[\int_{0}^{t}(F(B_s)-F(B_{s_n}))\delta B_s\\
&&+c_H\int_{0}^{t}\int_0^s\left(F'(B_s)\mathbb{I}_{[0,s]}(l)-F'(B_{s_n})\mathbb{I}_{[0,s_n]}(l)\right)|l-s|^{2H-2}dlds\Big]\\
&&-\frac{1}{2}\int_{0}^{t}F'(B_s)ds\\
&=&\frac{1}{2}\int_{0}^{t}\left(F'(B_{s_n})-F'(B_s)\right)ds\\
&&+n^{2H-1}c_H\int_{0}^{t}\int_0^{s_n}\left(F'(B_{s})-F'(B_{s_n})\right)|l-s|^{2H-2}dlds\\
&&+n^{2H-1}c_H\int_{0}^{t}\int_{s_n}^s\left(F'(B_{s})-F'(B_{s_n})\right)|l-s|^{2H-2}dlds\\
&&+n^{2H-1}\int_{0}^{t}(F(B_s)-F(B_{s_n}))\delta B_s,
\end{eqnarray*}
where we used the fact that $c_H\int_{t_n}^{t_{n+1}}\int_{t_n}^s|l-s|^{2H-2}dlds=\frac{1}{2}n^{-2H}$.

We have 
\begin{eqnarray*}
&&\mathbb{E}\left[\left(n^{2H-1}c_H\int_{0}^{t}\int_0^{s_n}\left(F'(B_{s})-F'(B_{s_n})\right)|l-s|^{2H-2}dlds\right)^2\right]\\
&\leq&Kn^{4H-2}\mathbb{E}\left[\left(\int_{0}^{t}(F'(B_s)-F'(B_{s_n}))ds\right)^2\right]
\end{eqnarray*}
and
\begin{eqnarray*}
&&\mathbb{E}\left[\left(n^{2H-1}c_H\int_{0}^{t}\int_{s_n}^s\left(F'(B_{s})-F'(B_{s_n})\right)|l-s|^{2H-2}dlds\right)^2\right]\\
&=&n^{4H-2}c^2_H\int_{0}^{t}\int_{0}^{t}\mathbb{E}\left[(F'(B_{s})-F'(B_{s_n})(F'(B_{x})-F'(B_{x_n})\right]\\
&&\times\left(\int_{s_n}^s\int_{x_n}^x|\theta-s|^{2H-2}|\mu-x|^{2H-2}d\mu d\theta\right)dsdx\\
&\leq&\mathbb{E}\left[\left(\int_{0}^{t}(F'(B_s)-F'(B_{s_n}))ds\right)^2\right],\\
\end{eqnarray*}
and
\begin{eqnarray*}
&&\mathbb{E}\left[\left(n^{2H-1}\int_{0}^{t}(F(B_s)-F(B_{s_n}))\delta B_s\right)^2\right]\\
&\leq&Kn^{4H-2}\mathbb{E}\left[\int_{0}^{t}(F(B_s)-F(B_{s_n}))^2ds\right]\\
&&+Kn^{4H-2}\mathbb{E}\left[\int_0^s\int_0^s(D_lu_s-D_lu_{s_n})^2dlds\right]\\
&\leq&Kn^{4H-2}\mathbb{E}\left[\int_{0}^{t}(F(B_s)-F(B_{s_n}))^2ds\right]\\
&&+Kn^{4H-2}\left(\mathbb{E}\left[\int_{0}^{t}\int_0^{s_n}(F'(B_s)-F'(B_{s_n}))^2dlds\right]\right.\\
&&\left.+\mathbb{E}\left[\int_{0}^{t}\int_{s_n}^sF'(B_s)^2dlds\right]\right)
\end{eqnarray*}
with $K$ depending only on $T$. We used  \eqref{embed} then \eqref{Meyer} in Proposition \ref{boundsSkorokhod}, and the fact that $D_lu_s-D_lu_{s_n}=F'(B_s)\mathbb{I}_{l\leq s}-F'(B_{s_n})\mathbb{I}_{l\leq s_n}$ to obtain the last inequality.
\medskip

We have $$n^{4H-2}\mathbb{E}\left[\int_{0}^{t}\int_{s_n}^sF'(B_s)^2dlds\right]\leq \frac{n^{4H-3}}{2}\mathbb{E}\left[\int_{0}^{t}F'(B_s)^2ds\right]\underset{n\rightarrow\infty}\longrightarrow0$$
(because $H<\frac{2}{3}$).
We have
\begin{eqnarray*}
&&\mathbb{E}\left[\int_{0}^{t}\int_0^{s_n}(F'(B_s)-F'(B_{s_n}))^2dlds\right]\\
&=&\mathbb{E}\left[\int_{0}^{t}\int_0^{s_n}(F'(B_s)-F'(B_{s_n}))^2\mathbb{I}_{\lfloor B_s\rfloor=\lfloor B_{s_n}\rfloor}dlds\right]\\
&&+\mathbb{E}\left[\int_{0}^{t}\int_0^{s_n}(F'(B_s)-F'(B_{s_n}))^2\mathbb{I}_{\lfloor B_s\rfloor\neq\lfloor B_{s_n}\rfloor}dlds\right].
\end{eqnarray*}
Using (\ref{deriv}), we have:
\begin{eqnarray*}
&&\mathbb{E}\left[\int_{0}^{t}\int_0^{s_n}(F'(B_s)-F'(B_{s_n}))^2\mathbb{I}_{\lfloor B_s\rfloor=\lfloor B_{s_n}\rfloor}dlds\right]\\
&=&\frac{1}{4}\mathbb{E}\left[\int_{0}^{t}\int_0^{s_n}dlds\mathbb{I}_{[-k-1,-k]\cup[k,k+1]}(B_s)\right.\\
&&\left.\left(\int_{-k-1}^{k+1}({\rm sign}(B_s-a)-{\rm sign}(B_{s_n}-a))dF''(a)\right)^2\mathbb{I}_{\lfloor B_s\rfloor=\lfloor B_{s_n}\rfloor}\right].
\end{eqnarray*}
Moreover, since $dF''$ is $\sigma$-finite, we can use Fubini's theorem and Jensen's inequality to get that
\begin{equation*}
\mathbb{E}\left[\int_{0}^{t}\int_0^{s_n}(F'(B_s)-F'(B_{s_n}))^2dlds\right]\leq t(C^n_{t}+D^n_{t}),
\end{equation*}
with
\begin{eqnarray*}
&C^n_{t}&= \sum_{k=0}^{+\infty}F''([-k-1,k+1))\\
&&\times \int_{-k-1}^{k+1}\int_0^t\mathbb{E}\left[\left({\rm sign}(B_s-a)-{\rm sign}(B_{s_n}-a)\right)^2\right.\\
&&\times \left.\mathbb{I}_{[-k-1,-k)\cup[k,k+1)]}(B_s)\mathbb{I}_{\lfloor B_s\rfloor=\lfloor B_{s_n}\rfloor}\right]dsdF''(a) ;\\
&D^n_{t}&= \mathbb{E}\left[\int_{0}^{t}(F'(B_s)-F'(B_{s_n}))^2\mathbb{I}_{\lfloor B_s\rfloor\neq\lfloor B_{s_n}\rfloor}ds\right].
\end{eqnarray*}
Let $\gamma>0$ and let $p,q>0$ be two conjugate exponents such that $\frac{H-\gamma}{p}>4H-2$. (Notice that $\gamma,p,q$ exist if and only if $H<\frac{2}{3}$.) We apply H\"older's inequality to obtain:
\begin{eqnarray*}
&&\mathbb{E}\left[\left({\rm sign}(B_s-a)-{\rm sign}(B_{s_n}-a)\right)^2\mathbb{I}_{[-k-1,-k)\cup[k,k+1)]}(B_s)\mathbb{I}_{\lfloor B_s\rfloor=\lfloor B_{s_n}\rfloor}\right]\\
&\leq&\mathbb{E}\left[\left|{\rm sign}(B_s-a)-{\rm sign}(B_{s_n}-a)\right|^{2p}\right]^{\frac{1}{p}}\mathbb{E}\left[\mathbb{I}_{[-k-1,-k)\cup[k,k+1)]}(B_s)\right]^{\frac{1}{q}}.
\end{eqnarray*}
We know that $\mathbb{E}\left[\mathbb{I}_{[-k-1,-k)\cup[k,k+1)]}(B_s)\right]^{\frac{1}{q}}=O_{k\rightarrow\infty}(e^{-\frac{k^2}{2qT^{2H}}})$ for all $s\in [0,T]$. We also have that (by hypothesis on $F''$),
$$\sum_{k=0}^{\infty}F''([-k-1,k+1))e^{-\frac{k^2}{2qT^{2H}}}=O_{k\rightarrow\infty}(1).$$
By Lemma \ref{boundOnSign},
for all $a\in\mathbb{R},$
$$\int_0^t\mathbb{E}\left[\left({\rm sign}(B_s-a)-{\rm sign}(B_{s_n}-a)\right)^{2p}\right]^{\frac{1}{p}}ds=o_{n\rightarrow\infty}(n^{2-4H}),$$
where the $o$ does not depend on $t$.

We then obtain:
\begin{equation}\label{boundC}n^{4H-2}C_{t}^n=o_{n\rightarrow\infty}(1).
\end{equation}
A similar use of Lemma \ref{boundOnSign} shows that, for all $t$,
\begin{eqnarray*}
n^{4H-2}D_{t}^n+n^{4H-2}\mathbb{E}\left[\left(\int_{0}^{t}(F(B_s)-F(B_{s_n}))ds\right)^2\right]
=o_{n\rightarrow\infty}(1).\label{boundD}
\end{eqnarray*}
Putting these facts together leads to:
$$M^n_{t}\overset{L^2(\Omega)}{\underset{n\rightarrow\infty}{\longrightarrow}}\frac{1}{2}\int_0^{t}F'(B_s)ds.$$

\end{proof}

\section{Proofs of the main Theorems and other results}

Throughout all this section, we denote by $B$ a fractional Brownian motion of Hurst index $H$.

\subsection{Miscellaneous} 
We start by giving a collection of technical results that are used throughout the paper.

The following lemma is an easy consequence of Fernique's theorem (see e.g. \cite{talagrand} and the references therein), and represents a very useful tool for proving the existence of moments for H\"older modulus of Gaussian functionals.

\begin{lemma}\label{Fernique} \textbf{ (Fernique) } 
Assume that $H>\frac{1}{2}$ and let $\mathbb{B}$ be the associated L\'evy area of $B$, defined as $\mathbb{B}_{s,t}^{k,j}=\int_s^t(B^k_l-B^k_s)dB_l^j$.
For all $\gamma\in(0,2)$ and all $\kappa\in(0,H)$, and for all function $f$ satisfying the growth condition $|f(x)|\leq\exp{|x|^{\gamma}}$, we have
\begin{equation*}
\mathbb{E}[f(\|B\|_{\kappa}+\|\mathbb{B}\|_{2\kappa})]<\infty,
\end{equation*}
where $\|\cdot\|_\theta$ is the H\"older seminorm, see (\ref{thetaholder})-(\ref{cc}).
\end{lemma}

We also have the following elementary lemma.
\begin{lemma}\label{simpleInequality}\textit{ Assume $H>\frac{1}{2}$.
There exists a constant $k_T>0$ such that, for all $x,y,s,t\in[0,T]^4$ such that $t\geq s$ and $y\geq x$,
\begin{equation}\label{jointIncrements}
k_T|t-s||y-x|\leq r_H(s,t,x,y)\leq |t-s|^H|x-y|^H,
\end{equation}
\begin{equation}\label{singleIncrement}
|\mathbb{E}[B_t(B_x-B_y)]|\leq |x-y|.
\end{equation}
}\end{lemma}
\begin{proof} For the sake of simplicity, we will consider $T=1$ (which only modifies the constants). In the expression \eqref{jointIncrements}, the right inequality is a simple consequence of the Cauchy-Schwarz inequality. 
For the proof of the left inequality, six cases must be analyzed carefully.

\begin{enumerate}
\item[(i)] case where $t\geq s\geq y\geq x$. For fixed $s,y,x$, let $$f(t)=(2H-1)(t-s)(y-x)\,\mbox{ and }\, g(t)=r_H(s,t,x,y).$$
We have $f(s)=g(s)=0$ and 
\begin{equation*} g'(t)-f'(t)=2H\left(-(t-y)^{2H-1}+(t-x)^{2H-1}\right)-(2H-1)(y-x).
\end{equation*}
We see that $-(t-y)^{2H-1}+(t-x)^{2H-1}\geq (2H-1)(y-x)$ thanks to an elementary function study, so $$g'(t)-f'(t)\geq (2H-1)^2(y-x)\geq 0,$$
so $g(t)\geq f(t)$ and then $(2H-1)|t-s||y-x|\leq r_H(s,t,x,y)$.\\

\item[(ii)]  case where $t\geq y>s\geq x$. For fixed $t,y,x$, we see (thanks to an elementary function study) that the quantity  $r_H(s,t,x,y)-(t-s)(y-x)$ decreases with $s$ and then reaches its minimum for $s=y$. Assume then $s=y$ and let $\delta=t-x$ and $a=y-x$. Then
\begin{eqnarray*} 
&r_H(s,t,x,y)-(t-s)(y-x)\\&\geq h(a)=\delta^{2H}-(a^{2H}+(\delta-a)^{2H}+a(\delta-a)).
\end{eqnarray*}
We have $h(\delta)=h(0)=0$, and the function $h$ is increasing over $(0,\frac{\delta}{2})$ then decreasing, so is always positive.\\

\item[(iii)]  case $y>t\geq x>s$. This is similar to (ii).\\

\item[(iv)]  case $y\geq x>t\geq s$. This is similar to (i).\\

\item[(v)]  case $t\geq y\geq x>s$. Write $B_t-B_s=(B_t-B_y)+(B_y-B_x)+(B_x-B_s)$ and then combine the inequalities from (i) and (iv).\\

\item[(vi)]  case $y>t\geq s\geq x$. This is similar to (v).\\

\end{enumerate}

Finally, the proof of inequality \eqref{singleIncrement} can be found in \cite[Lemma 6]{nourdin2010central}.
 \end{proof}

The following lemma is used in Step 2 of the proof of Proposition \ref{convergenceWienerIto}.
\begin{lemma}\label{expectProduct}
Let $m,n\in\mathbb{N}$ with $m>n$, $f\in\mathcal H^{\odot m}\otimes \mathcal H$, $g\in\mathcal H^{\odot n}$, $h\in\mathcal H^{\odot n}\otimes \mathcal H$. Let $x\in [0,T]$, $F_x=\delta^m(f(\cdot,x)), G=\delta^n(g), H_x=\delta^n(h(\cdot,x))$. Then, for all $s\leq t$ and $u\leq v$,
\begin{equation}\label{e1}\mathbb{E}[\delta(F_{\cdot}\mathbb{I}_{[s,t]}(\cdot))G]=0\end{equation}
\begin{equation}\label{e2}\mathbb{E}[\delta(F_{\cdot}\mathbb{I}_{[s,t]}(\cdot))\delta(H_{\cdot}\mathbb{I}_{[u,v]}(\cdot))]=0.
\end{equation}
\end{lemma}
\begin{proof}
 If $n=0$, the result is immediate. Otherwise, thanks to Proposition \ref{boundsSkorokhod}, we can write:
\begin{eqnarray*}
&&\mathbb{E}[\delta(F_{\cdot}\mathbb{I}_{[s,t]}(\cdot))G]\\
&=&\int_{[0,T]^2}\mathbb{E}[F_x\mathbb{I}_{[s,t]}(x)\delta^{n-1}(g(\cdot,y))]|x-y|^{2H-2}dxdy\\
&&+\int_{[0,T]^4}\mathbb{E}[D_w(F_x\mathbb{I}_{[s,t]}(x))D_z(\delta^{n-1}(g(\cdot,y)))]|x-y|^{2H-2}\\
&&\hskip6cm \times |w-z|^{2H-2}dxdydwdz.
\end{eqnarray*}
Thanks to Proposition \ref{multipleWiener}, we have, for all $x,y\in[0,T]$, $$\mathbb{E}[F_x\delta^{n-1}(g(\cdot,y))]=0.$$
Moreover, $D_w(F_x\mathbb{I}_{[s,t]}(x))=m\delta^{m-1}(f(\cdot,w,x))\mathbb{I}_{[s,t]}(x)$, and\\ $D_z(\delta^{n-1})(g(\cdot,y))=(n-1)\delta^{n-2}(g(\cdot,z,y))$ if $n\geq 2$ and $D_z(\delta^{n-1}(g(\cdot,y)))=0$ otherwise. In any case,  we have thanks to Proposition \ref{multipleWiener},
$$\mathbb{E}[D_w(F_x\mathbb{I}_{[s,t]}(x))D_z(\delta^{n-1}(g(\cdot,y)))]=0.$$
Equality (\ref{e2}) can be obtained by the same way.
\end{proof}

The following lemma provides a tightness criterion for two sequences of processes in $\mathcal{D}([0,T])$ whose sum belongs to $\mathcal{C}([0,T])$.
Recall the notation $s_n=\frac1n\lfloor ns\rfloor$ and $t_n=\frac1n\lfloor nt\rfloor$.
\begin{lemma}\label{tightnessCriterion}
Let $(X^n)\subset\mathcal C_{\mathbb{R}}([0,T])$ be a sequence of continous stochastic processes such that $X^n_t=A^n_t+C^n_t$ for all $t\in [0,T]$, where $A^n,C^n\in\mathcal{D}_{\mathbb{R}}([0,T])$.
Assume also the existence of 
$\alpha_0,\beta_0>0$ such that
\begin{equation}
\mathbb{E}\left[|A_t^n-A^n_s|^{\beta_0}\right]\leq K|t_n-s_n|^{1+\alpha_0}, \quad 0\leq s,t\leq T,
\end{equation}
and
\begin{equation}
\sup_{t\in [0,T]}|C^n_t|\overset{\mathbb{P}}{\underset{n\rightarrow\infty}\longrightarrow} 0.
\end{equation} 
Then the sequence $X^n$ is tight in $\mathcal C_{\mathbb{R}}([0,T])$.
\end{lemma}
\begin{proof}
In \cite{dbler2019functional}, it is proved that the sequence $(A^n)$ is tight in $\mathcal{D}([0,T])$. Moreover, the sequence $(C^n)$ is also tight in $\mathcal{D}([0,T])$. As a result,  $(X^n)$ is tight in $\mathcal{D}([0,T])$ as a sum of two tight sequences. Since the uniform and the Skorohod topologies coincide on $\mathcal{C}([0,T])$, we deduce that $(X^n)$ is tight in $\mathcal C([0,T])$.
\end{proof}

The following lemma is used in the proof of the forthcoming (\ref{convconv2}).

\begin{lemma}\label{equiContinuity}
Let $(X^n)\subset\mathcal C_{\mathbb{R}}([0,T])$ be a tight sequence of continous stochastic processes such that $\forall t\in [0,T]$, $X^n_t\overset{\mathbb{P}}{\longrightarrow}0$ as $n\to\infty$. Then, as $n\to\infty$,
\begin{equation}\sup_{t\in[0,T]}|X^n_t|\overset{\mathbb{P}}{\longrightarrow}0.
\end{equation}
\end{lemma}
\begin{proof} Let $1>\epsilon>0$. Since the function defined  by $x\rightarrow 1\wedge \sup_{t\in [0,T]}|x_t|$ on $\mathcal{C}([0,T])$ is continuous and bounded, we deduce that $\mathbb{E}\left[1\wedge \sup_{t\in [0,T]}|X^n_t|\right]\to 0$ as $n\to\infty$. Then, by Markov's inequality
and as $n\to\infty$,
$$\mathbb{P}\left[\sup_{t\in [0,T]}|X^n_t|>\epsilon\right]= \mathbb{P}\left[\sup_{t\in [0,T]}|X^n_t|\wedge 1>\epsilon\right]\leq \frac{1}{\epsilon}\mathbb{E}\left[\sup_{t\in [0,T]}|X^n_t|\wedge 1\right]\longrightarrow 0.$$
\end{proof}

The following lemma gives technical estimates used in the proof of Propositions \ref{brownianCriterion} and \ref{prp}.

\begin{lemma}\label{boundOnSign} 
Assume $H>\frac12$. Then, for all $0\leq t\leq T$, $a\in\mathbb{R}$, $\gamma>0$, $p>0$ and $\theta\geq1$, we have
\begin{equation}\label{boundSign1} 
 \int_{0}^{t}\mathbb{E}\left[\left|{\rm sign}(B_s-a)-{\rm sign}(B_\frac{\lfloor ns \rfloor}{n}-a)\right|^p\right]^{\frac{1}{\theta}}ds
\leq Kn^{\frac{-H+\gamma}{\theta}},
\end{equation}
and
\begin{equation}
 \int_{0}^{t}\mathbb{E}\left[\mathbb{I}_{\lfloor B_s\rfloor\neq \lfloor B_{s_n}\rfloor}\right]^{\frac{1}{\theta}}ds\leq Kn^{\frac{-H+\gamma}{\theta}}\quad\mbox{(where  $s_n=\frac1n\lfloor ns\rfloor$)},
\end{equation}
with $K$ depending only on $T,\delta,p,\theta$ and where $sign$ is the left derivative of the function $x\rightarrow|x|$.
\end{lemma}
\begin{proof}
We only do the proof of the first inequality, the proof of the second one being similar. Moreover, for simplicity we reduce to $a=0$ and $\theta=1$.

We have
\begin{equation*}
{\rm sign}(B_s)-{\rm sign}(B_\frac{\lfloor ns\rfloor}{n})=2\mathbb{I}_{\left\{B_s\geq0,B_\frac{\lfloor ns\rfloor}{n}<0\right\}}-2\mathbb{I}_{\left\{B_s<0,B_\frac{\lfloor ns\rfloor}{n}\geq0\right\}}.
\end{equation*}
Plugging this identity into the integral yields
\begin{equation*}
\int_{0}^{t}\mathbb{E}\left[\left|{\rm sign}(B_s)-{\rm sign}(B_\frac{\lfloor ns \rfloor}{n})\right|^p\right]ds\leq 2^{p+1}\int_{0}^{t}\mathbb{E}\left[\mathbb{I}_{\{B_s\geq0,B_\frac{\lfloor ns \rfloor }{n}<0\}}\right]ds.\end{equation*}
On the other hand, for all $k\in\{2,\ldots, nT_n\}$ and $s\in[\frac{k}{n},\frac{k+1}{n}\wedge T),$
\begin{itemize}
\item if $s=\frac{k}{n}$, then $\mathbb{P}\left[\mathbb{I}_{\{B_s\geq0,B_{s_n}<0\}}\right]=0$
\item else,\end{itemize}

\begin{eqnarray*}
&&\mathbb{P}\left[\mathbb{I}_{\{B_s\geq0,B_{s_n}<0\}}\right]=\mathbb{P}\left[\mathbb{I}_{{B_\frac{k}{n}<0,\,\,B_s-B_\frac{k}{n}\geq-B_\frac{k}{n}}}\right]
\\ &\leq&\sum_{i=0}^{+\infty}\mathbb{E}\left[\mathbb{I}_{\left\{B_\frac{k}{n}\in[-\frac{i+1}{n},-\frac{i}{n}]\right\}}\mathbb{I}_{\left\{B_s-B_\frac{k}{n}>\frac{i}{n}\right\}}\right]
\\ &=&\sum_{i=0}^{+\infty}\mathbb{E}\left[\mathbb{I}_{\left\{B_1\in\left[-\frac{i+1}{n^{1-H}~k^H},-\frac{i}{n^{1-H}~k^H}\right)\right\}}\mathbb{I}_{\left\{B_\frac{ns}{k}-B_1>\frac{i}{n^{1-H}~k^H}\right\}}\right]\\
&&\mbox{(using the self-similarity of $B$)}\\
&=&\frac{1}{2\pi\sqrt{\det(\Sigma)}}\sum_{i=0}^{+\infty}\int_\frac{i}{k^H~n^{1-H}}^\frac{i+1}{k^H~n^{1-H}}\int_{-\infty}^\frac{-i}{k^Hn^{1-H}}e^{-\frac{(\frac{ns}{k}-1)^{2H}}{2\det(\Sigma)}x^2-\frac{1}{2\det(\Sigma)}y^2+\frac{c_{s,k}}{\det(\Sigma)}xy}dydx\\
&&\mbox{
with $
c_{s,k}=\mathbb{E}\left[B_1\left(B_\frac{ns}{k}-B_1\right)\right]>0 $,
$\Sigma =\begin{pmatrix}
1 & c_{s,k} \\
c_{s,k} & \left(\frac{ns}{k}-1\right)^{2H}
\end{pmatrix}
$
}.
\end{eqnarray*}
We have $\det(\Sigma)=(\frac{ns}{k}-1)^{2H}-c_{s,k}^2$. According to Lemma \ref{simpleInequality}, $c_{s,k}\leq \frac{ns}{k}-1$. If $k>2$, we have $\frac{ns}{k}-1\leq\frac{1}{2}$ and then \begin{eqnarray*}\left(\frac{ns}{k}-1\right)^{2H}&\geq&\det(\Sigma)\geq \left(\frac{ns}{k}-1\right)^{2H}\left(1-\left(\frac{ns}{k}-1\right)^{2-2H}\right)\\&\geq&\left(\frac{ns}{k}-1\right)^{2H}\left(1-\frac{1}{2^{2-2H}}\right).\end{eqnarray*}
For all $(x,y)\in\left[\frac{i}{k^H~n^{1-H}},\frac{i+1}{k^H~n^{1-H}}\right]\times\left(-\infty, \frac{-i}{k^Hn^{1-H}}\right]$, $\frac{c_{s,k}}{\det(\Sigma)}xy\leq 0$. Then, 

\begin{eqnarray}
&&\mathbb{P}\left[\mathbb{I}_{\{B_s\geq0,B_{s_n}<0\}}\right]\notag\\
&\leq&\frac{1}{2\pi\sqrt{(1-2^{2H-2})}(\frac{ns}{k}-1)^{H}}\int_\frac{i}{k^Hn^{1-H}}^\frac{i+1}{k^Hn^{1-H}}\int_{-\infty}^\frac{-i}{k^Hn^{1-H}}\notag\\
&&\hskip4cm\times e^{-\frac{x^2}{2}-\frac{1}{2(\frac{ns}{k}-1)^{2H}}y^2}dydx\notag\\
&\leq& \frac{1}{2\pi n^{1-H}k^H\sqrt{(1-2^{2H-2})}(\frac{ns}{k}-1)^{H}}\notag\\
&&\hskip4cm\times \sum_{i=0}^{+\infty}e^{-(\frac{i}{k^Hn^{1-H}})^2}\int_{-\infty}^\frac{-i}{k^Hn^{1-H}}e^{-\frac{1}{2(\frac{ns}{k}-1)^{2H}}y^2}dy\notag\\
&\leq& \frac{1}{2\pi n^{1-H}k^H\sqrt{(1-2^{2H-2})}(\frac{ns}{k}-1)^{H}}\sum_{i=0}^{+\infty}\notag\\
&&\hskip4cm\times e^{-\frac{1}{2}(\frac{i}{k^Hn^{1-H}})^2}\int_{-\infty}^\frac{-i}{k^Hn^{1-H}}e^{-k^{2H}y^2}dy\notag\\
&\leq& \frac{1}{2\pi n^{1-H}k^{2H}\sqrt{(1-2^{2H-2})}(\frac{ns}{k}-1)^{H}}\sum_{i=0}^{+\infty}e^{-\frac{1}{2}(\frac{i}{k^Hn^{1-H}})^2}\int_{-\infty}^\frac{-i}{n^{1-H}}e^{-y^2}dy\notag\\
& \leq &\frac{1}{2\pi n^{1-H}k^{2H}\sqrt{(1-2^{2H-2})}(\frac{ns}{k}-1)^{H}}\sum_{i=0}^{+\infty}e^{-\frac{1}{2}(\frac{i}{k^Hn^{1-H}})^2}\int_{-\infty}^\frac{-i}{n^{1-H}}e^{-y^2}dy.\notag\\
\label{previousexpression}
\end{eqnarray}
We have, for all $ i\in\mathbb{N}^*$, for all $\alpha\in (0,1)$, 
\begin{equation*} 
e^{-\frac{1}{2}(\frac{i}{k^Hn^{1-H}})^2}\leq \frac{(n^{1-H}k^H)^{\alpha}}{i^\alpha} \text{ and } \int_{-\infty}^\frac{-i}{n^{1-H}}e^{-y^2}dy\leq \frac{n^{1-H}}{i}.
\end{equation*}
Let $\alpha$ be such that $\alpha(2-H)<\gamma$. The right-hand side (\ref{previousexpression}) is then less than
\begin{equation*}
\frac{n^{\alpha(1-H)}}{2\pi k^{2H-\alpha}\sqrt{(1-2^{2H-2})}(\frac{ns}{k}-1)^{H}}\sum_{i=1}^{\infty}\frac{1}{i^{1+\alpha}}.
\end{equation*}

Finally,
\begin{eqnarray*}
&&\int_{0}^{t}\mathbb{E}\left[\left|{\rm sign}(B_s)-{\rm sign}(B_\frac{\lfloor ns \rfloor}{n})\right|^p\right]ds\\
&\leq&2^{p+1}\left(\frac{2}{n}+\int_{\frac{2}{n}}^{t}\mathbb{P}[B_s\geq0, B_{s_n}<0]ds\right)\\
\end{eqnarray*}
and 
\begin{eqnarray*}
&&\int_{\frac{2}{n}}^{t}\mathbb{P}[B_s\geq0, B_{s_n}<0]ds\\
&\leq&\frac{1}{2\pi}\sum_{k=1}^{nt_{n}} \frac{1}{k^{2H-\alpha}}\int_{\frac{k}{n}}^{\frac{k+1}{n}\wedge t}\frac{1}{\sqrt{(1-2^{2H-2})}(\frac{ns}{k}-1)^{H}}ds\sum_{i=1}^{\infty}\frac{1}{i^{1+\alpha}} n^{\alpha(1-H)}\\
&\leq& \frac{1}{2\pi\sqrt{(1-2^{2H-2})}}\left(\frac{(t-t_n)^H}{(nt_n)^{2H-\alpha}}+\frac{1}{1-H}\sum_{k=2}^{nt_n} \frac{1}{nk^{H-\alpha}}\right)\sum_{i=1}^{\infty}\frac{1}{i^{1+\alpha}} n^{\alpha(1-H)}\\
&\leq& Kn^{-H+\alpha(2-H)}
\end{eqnarray*}
This provides the desired estimate.
\end{proof}

\bigskip
\subsection{Weighted quadratic variations of the fractional Brownian motion}

In the proofs of Theorems \ref{main1} and \ref{main2}, we will see that the announced convergences are determined by the asymptotic behaviour of the weighted quadratic variations of the fractional Brownian motion. 
These variations have already been extensively studied, for example in \cite{hu2016rate,nourdin2010central} and especially in \cite{corcuera2014asymptotics}. 
In the next three lemmas, we gather the results that are relevant to us, and we extend them when
necessary. 

\begin{lemma}\label{quadraticVariation} 
Let $x$ be a scalar process over $[0,T]$,
and assume it is a.s continuous and satisfies 
$\mathbb{E}\left[\|x\|^{2+\gamma}_{\infty}\right]<+\infty$
for some $\gamma>0$. Let $H>\frac{1}{2}$. Then,
\begin{enumerate}
\item For all $j\leq d$, for all $t\in [0,T]$
\begin{equation}\label{convconv1}
S^{n,j}_{t,x}=n^{2H-1}\sum_{k=0}^{\lfloor nt\rfloor }x_\frac{k}{n}\left(B^{j}_{\frac{k+1}{n}\wedge t}-B^{j}_\frac{k}{n}\right)^2\overset{L^2(\Omega)}{\underset{n\rightarrow\infty}{\longrightarrow}}\int_0^tx_sds.
\end{equation}
\item For all $i\neq j$, 
\begin{equation}\label{convconv2}
n^{2H-1}\sum_{k=0}^{\lfloor nt\rfloor }x_\frac{k}{n}\delta^{1,i}\left(\left(B^{j}_.-B^{j}_\frac{k}{n}\right)\mathbb{I}_{[\frac{k}{n},\frac{k+1}{n}\wedge t](.)}\right)\overset{L^2(\Omega)}{\underset{n\rightarrow\infty}\longrightarrow}0.
\end{equation}
\end{enumerate}
These two convergences also holds UCP as a process over $[0,T]$.
\end{lemma}
\begin{proof}
\textit{\underline{Step 1: Proof of (\ref{convconv1})}}.
It is well known that (\ref{convconv1}) is true in the a.s. sense if $x=\mathbb{I}_{[0,t]}$ (see e.g \cite{klein1975quadratic}) and then (by substraction) for every process of the type $x=\mathbb{I}_{[s,t]}$ for $s\leq t$. Now, consider $0\leq a_0\leq\ldots\leq a_p\leq T$ and let $(\alpha_0,\ldots ,\alpha_{p-1})$ be a collection of $\mathfrak F$-measurable random variables. For all $1\leq i\leq p$, let $\Omega_i$ be the subset of $\Omega$ on which (\ref{convconv1}) holds true for the process $\alpha_i\mathbb{I}_{[a_i,a_{i+1}]}$. Then $\mathbb{P}(\cap_{i=1}^p \Omega_i)=1$, and $(\ref{convconv1})$ holds (pointwise) for the step process $x=\sum_{i=0}^{p-1}\alpha_i\mathbb{I}_{[a_i,a_{i+1}]}$ on $\cap_{i=1}^p \Omega_i$.

Moreover, if a process $f$ is bounded for $\|\cdot\|_\infty$ in $L^{2+\gamma}$ then  the sequence $\{(S^{n,j}_{t,f})^2\}_{n=1}^\infty$ is uniformly integrable. Indeed, let $A\in\mathfrak F$ and $0<\mu<\gamma$. Then
\begin{eqnarray*}
&&\mathbb{E}\left[\mathbb{I}_A(S_{t,f}^{n,j})^2\right]\\
&=& n^{4H-2}\mathbb{E}\left[\sum_{k,l=0}^{\lfloor nt\rfloor }\mathbb{I}_Af_\frac{k}{n}f_\frac{l}{n}\left(B^j_{\frac{k+1}{n}\wedge t}-B^j_\frac{k}{n}\right)^2\left(B^j_{\frac{l+1}{n}\wedge t}-B^j_\frac{l}{n}\right)^2\right]\\ 
&\leq& n^{4H-2}\sum_{k,l=0}^{\lfloor nt\rfloor }\left(\mathbb{E}\left[\mathbb{I}_A|f_\frac{k}{n}|^{1+\frac{\mu}{2}}|f_\frac{l}{n}|^{1+\frac{\mu}{2}}\right]\right)^\frac{1}{1+\frac{\mu}{2}}
\\
&&\hskip1cm \times \left(\mathbb{E}\left[\left(B^j_{\frac{k+1}{n}\wedge t}-B^j_\frac{k}{n}\right)^{2q}\left(B^j_{\frac{l+1}{n}\wedge t}-B^j_\frac{l}{n}\right)^{2q}\right]\right)^\frac{1}{q}\\
&&\mbox{(using H\"older inequality with $q$ the conjugate of $1+\frac{\mu}{2}$)}\\
&\leq& K\sup_{s\in[0,T]}\left(\mathbb{E}[\mathbb{I}_A|f_s|^{2+\mu}]\right)^\frac{2}{2+\mu},
\end{eqnarray*}
and this quantity converges to $0$ as $\mathbb{P}(A)\rightarrow 0$, because $\left(|f_s|^{2+\mu}\right)_{s\in[0,T]}$ is uniformly integrable.

Back to the initial process $x$, we know, by uniform continuity of $x$ on $[0,T]$, that $\|x-x^{m}\|_{\infty}\underset{m\rightarrow\infty}\longrightarrow0$ a.s. (where $x^{m}$ is the sampled process $x_\frac{\lfloor m\cdot\rfloor}{m}$). As a result,
\begin{eqnarray*}
&&\mathbb{E}\left[\left(S^{n,j}_{t,x}-\int_0^tx_sds\right)^2\right] \\ 
&=&\mathbb{E}\left[\left((S^{n,j}_{t,x}-S^{n,j}_{t,x^m})+\left(S^{n,j}_{t,x^m}-\int_0^tx^m_sds\right)+\left(\int_0^tx^m_sds-\int_0^tx_sds\right)\right)^2\right]\\
&\leq& C\left\{n^{4H-2}\mathbb{E}\left[\left(\sum_{k=0}^{nt_n }\|x-x^{m}\|_{\infty}\left(B_{\frac{k+1}{n}\wedge t}-B_\frac{k}{n}\right)^2\right)^2\right]\right.\\
&&\left.\hskip1cm+\mathbb{E}\left[\left(S_{t,x^m}^{n,j}-\int_0^tx^{m}_sds\right)^2\right]+T^2\mathbb{E}\left[\|x-x^{m}\|_{\infty}^2\right]\right\},
\end{eqnarray*}
where $C$ is a positive constant.
The previous arguments, an appropriate choice of $n,m\in\mathbb{N}^*$ and the fact that $\|x-x^m\|_{\infty}\overset{L^{2+\gamma}(\Omega)}{\underset{n\rightarrow\infty}\longrightarrow}0$ allow to conclude.\\

\smallskip\textit{\underline{Step 2: UCP convergence of $S^{n,j}_{\cdot,x}$. }}
According to Lemma \ref{equiContinuity}, the UCP convergence of $S^{n,j}_{\cdot,x}$ to $\int_0^\cdot x_sds$ follows from the convergence in probability of $S^{n,j}_{t,x}$ for fixed $t$
and the tightness of the sequence $(S^{n,j}_{\cdot,x})_n$ in $\mathcal{C}([0,T])$. 
The convergence in probability for fixed $t$ is shown in Step 1. 
For the tightness, this can be 
checked with Lemma \ref{tightnessCriterion} applied to $S^{n,j}_{t,x}=A^n_t+C^n_t$, with
\begin{eqnarray*}
&&A^n_t=n^{2H-1}\sum_{k=0}^{nt_n-1}x_\frac{k}{n}\left(B^j_{(\frac{k+1}{n})\wedge t}-B^j_\frac{k}{n}\right)^2\\
&&C^n_t =n^{2H-1}x_{t_n}\left(B^j_{t}-B^j_{t_n}\right)^2\\
&&\alpha_0=1,\beta_0=2.
\end{eqnarray*}
Indeed, using the H\"older inequality, we have for $s\leq t$:
\begin{eqnarray*}
&&\mathbb{E}[|A^n_{t}-A^n_s|^2]\\
&\leq& n^{4H-2}\sum_{k,l=ns_n}^{nt_n-1}\left(\mathbb{E}\left[\|x\|^{2+\gamma}_{\infty}\right]\right)^\frac{1}{1+\frac{\gamma}{2}}\left(\mathbb{E}\left[\left(B^j_\frac{k+1}{n}-B^j_\frac{k}{n}\right)^{2p'}\left(B^j_\frac{l+1}{n}-B^j_\frac{l}{n}\right)^{2p'}\right]\right)^{\frac{1}{p'}}\\
&\leq& n^{4H}|t_n-s_n|^{2}\left(\mathbb{E}\left[\|x\|^{2+\gamma}_{\infty}\right]\right)^\frac{1}{1+\frac{\gamma}{2}}\left(\mathbb{E}\left[(B_\frac{1}{n})^{4p'}\right]\right)^\frac{1}{p'}\leq K|t_n-s_n|^{2}
\end{eqnarray*}
with $p'$ the conjugate of $1+\frac{\gamma}{2}$ and $K$ some constant depending only on $\gamma$ and $x$.

On the other hand, $B$ has $(H-\epsilon)$-H\"older continuous paths for every $\epsilon>0$, so that, for all $t\in [0,T]$, $|C^n_t|\leq K_{\epsilon}n^{2\epsilon-1}\|x\|_{\infty}$ a.s. for some random variable $K_\epsilon>0$. Taking $\epsilon$ small enough, we have 
$\sup_{t\in [0,T]} |C^n_t| \overset{\mathbb{P}}{\underset{n\rightarrow\infty}\longrightarrow}0$.

\textit{\underline{Step 3: Proof of (\ref{convconv2})}}.
We now turn to the case $i\neq j$. Similarly to the proof of (\ref{convconv1}) (Step 1), we first show (\ref{convconv2}) for $x$ the function identically one, in other words:
\begin{equation}\label{azerty}
S^{n,i,j}_{t,1}=n^{2H-1}\sum_{k=0}^{nt_n }\delta^{1,i}\left(\left(B^{j}_.-B^{j}_\frac{k}{n}\right)\mathbb{I}_{[\frac{k}{n},\frac{k+1}{n}\wedge t](\cdot)}\right)\overset{L^2(\Omega)}\longrightarrow0.
\end{equation}
Using Proposition \ref{boundsSkorokhod} and taking into account that $D^{1,i}B^{j}=0$ if $i\neq j$, we have:
\begin{eqnarray*}
&&\mathbb{E}\left[(S_{t,1}^{n,i,j})^2\right]\\ 
&=& n^{4H-2}\sum_{k,l=0}^{nt_n }\mathbb{E}\left[\left\langle\left(B^{j}_.-B^{j}_\frac{k}{n}\right)\mathbb{I}_{[\frac{k}{n},\frac{k+1}{n}\wedge t]}(.),\left(B^{j}_\cdot-B^{j}_\frac{l}{n}\right)\mathbb{I}_{[\frac{l}{n},\frac{l+1}{n}\wedge t]}(.)\right\rangle_{\mathcal H}\right]\\ 
&=&n^{4H-2}\sum_{k,l=0}^{nt_n}c_H\int_{\frac{k}{n}}^{\frac{k+1}{n}\wedge t}\int_{\frac{l}{n}}^{\frac{l+1}{n}\wedge t}\mathbb{E}[(B_x-B_{\frac{k}{n}})(B_y-B_\frac{l}{n})]|y-x|^{2H-2}dydx\\
&\leq& n^{2H-2}\sum_{k,l=0}^{nt_n}\langle\mathbb{I}_{[\frac{k}{n},\frac{k+1}{n}\wedge t]}(.),\mathbb{I}_{[\frac{l}{n},\frac{l+1}{n}\wedge t]}(\cdot)\rangle_{\mathcal H},
\end{eqnarray*}
where the last inequality follows from the fact that: for all $x\in\left[\frac{k}{n},\frac{k+1}{n}\right]$ and all $y\in\left[\frac{l}{n},\frac{l+1}{n}\right]$, \\$|\mathbb{E}[(B_x-B_{\frac{k}{n}})(B_y-B_\frac{l}{n})]|\leq \left(\mathbb{E}[(B_x-B_\frac{k}{n})^2]\right)^{\frac{1}{2}}\left(\mathbb{E}[(B_y-B_\frac{l}{n})^2]\right)^\frac{1}{2}\leq \frac{1}{n^{2H}}$.\\
We also have
\begin{equation*}
\sum_{k,l=0}^{nt_n }\langle\mathbb{I}_{[\frac{k}{n},\frac{k+1}{n}\wedge t]}(\cdot),\mathbb{I}_{[\frac{l}{n},\frac{l+1}{n}\wedge t]}(.)\rangle_{\mathcal H}=t^{2H},
\end{equation*}
and then $\mathbb{E}\left[(S_{t,1}^{n,i,j})^2\right]=O_{n\rightarrow\infty}(n^{2H-2})$, implying (\ref{azerty}).
To prove (\ref{convconv2}) in the general case for $x$, we can then proceed exactly as in the proof of (\ref{convconv1}), that is, we show first that \eqref{convconv2} holds true for step processes and then, by an approximation argument, to $x$. Tightness in $\mathcal{C}([0,T])$ can also be obtained as for (\ref{convconv1}). 
By Lemma \ref{equiContinuity}, this proves the UCP convergence to $0$ of each $S^{n,i,j}_{\cdot,x}$ with $i\neq j$.
\end{proof}

The study of the fluctuations (which are required for the proof of Theorem \ref{main2}) being more delicate, more stringent assumptions on the process $x$ are required (except when $H=\frac{1}{2}$, see the first point in the proposition below).

\begin{lemma} \label{skorokhodVariation} 
Let $x=(x^{i,e})_{1\leq i\leq m, 1\leq e\leq d}$ be an  $(m\times d)$-dimensional process,
and recall  the matrix-valued processes $W$ and $Z$ from Section \ref{rosenblatt-and-brownian}.
For any $1\leq i\leq m$ and $1\leq e,j\leq d$, set
\begin{equation*} 
S_{t,x}^{n,i,j,e}=\sum_{k=0}^{nt_n}x^{i,e}_\frac{k}{n}\int_\frac{k}{n}^{\frac{k+1}{n}\wedge t}\left(B^{e}_s-B^{e}_\frac{k}{n}\right)\delta B^{j}_s.
\end{equation*}
We have
\begin{enumerate}
\item If $H=\frac{1}{2}$ and if,  for all $(i,e)$,  $x^{i,e}$ is adapted to $B^e$, piecewise continuous and satisfies 
$\mathbb{E}[\sup_{i,e}\|x^{i,e}\|_{\infty}^{2+\gamma}]<\infty$ for some $\gamma>0$, then, stably in $\mathcal{C}_{\mathbb{R}^{d^2\times m}}([0,T])$,
\begin{equation}\label{convconv3}
\left(\sqrt{n}S_{\cdot,x}^{n,i,j,e}\right)_{i,j,e}\overset{}{\underset{n\rightarrow\infty}\longrightarrow}\left(\int_0^.x^{i,e}_sdW^{e,j}_{s}\right)_{i,j,e}.
\end{equation}
\item If $\frac{1}{2}< H\leq\frac{3}{4}$ and if $x$ is $\beta$-H\"older continuous for some $\beta>\frac{1}{2}$,
then, stably in $\mathcal{C}_{\mathbb{R}^{d^2\times m}}([0,T])$,
\begin{eqnarray}\label{convconv4}
\left(n^{2H-1}\nu_H(n)S_{\cdot,x}^{n,i,j,e}\right)_{i,j,e}\overset{}{\underset{n\rightarrow\infty}\longrightarrow}\left(\int_0^\cdot x^{i,e}_sdW^{e,j}_s\right)_{i,j,e}.
\end{eqnarray}
\item If $H>\frac{3}{4}$ and if $x$ verifies that 
$\mathbb{E}\left[\|x\|_{\beta}^{2+\gamma}\right]<+\infty$ for some $\beta>\frac{1}{2}$ and $\gamma>0$
then, in probability uniformly on $[0,T]$ (and also in $L^2(\Omega)$ for fixed $t$),
\begin{equation}\label{convconv5}
\left(n^{2H-1}\nu_H(n)S_{\cdot,x}^{n,i,j,e}\right)_{i,j,e}\overset{}
{\underset{n\rightarrow\infty}\longrightarrow}\left(\int_0^\cdot x^{i,e}_sdZ^{e,j}_s\right)_{i,j,e}.\end{equation}
\end{enumerate}
\end{lemma}

\begin{proof}
Even if they are not stated in exactly the same way, the limits (\ref{convconv4}) and (\ref{convconv5}) follow from \cite{corcuera2014asymptotics,hu2016rate}
(see especially \cite[Sections 4,5,7]{hu2016rate}) by means of fractional integration techniques.
This is why we only concentrate on the case $H=\frac12$ and the proof of \eqref{convconv3}, not covered by \cite{corcuera2014asymptotics,hu2016rate}.

\medskip

\textit{\underline{Proof of (\ref{convconv3}). }}
We divide the proof into three steps.
In the sequel, 'f.d.d.' is shorthand for finite dimensional distributions.
\medskip

\textit{\underline{Step 1: Convergence of the f.d.d. when $x$ is a step process}:} Let us first sketch the proof in the case where $x$ is constant over an interval, without going too much into the details, since the approach is very similar to that in \cite[Section 4]{hu2016rate}. Let $0\leq s\leq t\leq T$, let $q\in\mathbb{N}^*$, let $0=t_0\leq t_1\leq \ldots\leq t_q\leq t$ and let $x$ be the matrix function whose entries are all equal to $\mathbb{I}_{[s,t]}$. It is immediate that the $\mathbb{R}^{d^2m}$-valued random vector $$X_x^n=((\sqrt{n}S^{n,i,j,e}_{t_{l+1},x}-\sqrt{n}S^{n,i,j,e}_{t_l,x})_{i,j,e})_{l\in \{0,\ldots, q-1\}}$$ has independent entries. We can also check that 
$$\mathbb{E}[(X^n_{x})^{i_1,j_1,e_1}_l(X^n_{x})^{i_2,j_2,e_2}_l]=0$$ for all $i_1, i_2$, all $(j_1,e_1)\neq (j_2,e_2)$ and all $l\in \{0,\ldots,q-1\}$. Finally, we can easily show that $$\mathbb{E}[\left((X^n_{x})^{i_1,j_1,e_1}_l\right)^4]\underset{n\rightarrow\infty}\longrightarrow \frac{3}{4}(t_{l+1}\vee s-t_l\vee s)^2.$$ Peccati and Tudor's \textit{fourth moment theorem} \cite{peccati2005gaussian} applies, and shows the stable convergence $$X^n_x\overset{\mathcal{L}}{\underset{n\rightarrow\infty}\longrightarrow} \left(\left(W^{e,j}_{t_{l+1}\vee s}-W^{e,j}_{t_l\vee s})\right)_{i,j,e}\right)_l.$$ Since the increments are independent, this gives the convergence of the finite dimensional distributions in (\ref{convconv3}) when $x=\mathbb{I}_{[s,t]}$.

\medskip
Now, let $[s_1,t_1],\ldots,[s_q,t_q]$ be $q$ mutually disjoint intervals. Due to the independence of Brownian increments, the process $$\left(\sqrt{n}S^{n,i,j,e}_{\cdot,\mathbb{I}_{[s_1,t_1]}},\ldots,\sqrt{n}S^{n,i,j,e}_{\cdot,\mathbb{I}_{[s_q,t_q]}}\right)_{i,j,e}$$ has independent entries, so we have the stable convergence of its f.d.d.\! to the f.d.d.\! of the process $\left(\int_0^T\mathbb{I}_{[s_1,t_1]}dW^{e,j}_s,\ldots,\int_0^T\mathbb{I}_{[s_q,t_q]}dW^{e,j}_s\right)_{i,j,e}$. This implies in turn the convergence of the f.d.d.\! of $\sqrt{n}S^n_{\cdot,x}$ for processes $x$ of the form:
$$x=\sum_{l=0}^{q-1} F_l\mathbb{I}_{[t_{l},t_{l+1}]},$$
where $q\in\mathbb{N}^*$ and $F_l$ is a $\mathbb{R}^{m\times d}$ valued and $\mathfrak F_{t_l}$-measurable random variable.

\medskip
\textit{\underline{Step 2: Convergence of the f.d.d in the general case:}} We now turn to the general case. Let $x$ be an adapted, almost surely piecewise continuous process such that $\mathbb{E}[\sup_{i,e}\|x^{i,e}\|_{\infty}^2]<\infty$, and set $$\Delta_{e,j,k,n}(t)=\int_\frac{k}{n}^{\frac{k+1}{n}\wedge t}\left(B^{e}_s-B^{e}_\frac{k}{n}\right)\delta B^{j}_s.$$ As is the proof of Lemma \ref{quadraticVariation}, we can rely on the \textit{small blocks / big blocks} technique by considering the approximation 
\begin{equation}\label{deco}
\sqrt{n}S_{t,x}^{n,i,j,e}=\sqrt{n}S_{t,x^m}^{n,i,j,e}+\sqrt{n}\left(S_{t,x}^{n,i,j,e}-S_{t,x^m}^{n,i,j,e}\right)=\sqrt{n}S_{t,x^m}^{n,i,j,e}+R^{i,j,e}_{t,m,n,x},
\end{equation}
with $m\leq n$ and $x^m$ the sampled process $x_{\frac{\lfloor m\cdot\rfloor}{m}}$.\\
Fix $m\in\mathbb{N}^*$. Since $x^m$ is a step process, we have by Step 1  that $$
{\rm f.d.d.}-\lim_{n\rightarrow\infty}\left(\sqrt{n}S_{\cdot,x^m}^{n,i,j,e}\right)_{i,j,e}=\left(\int_0^\cdot (x^m)^{i,e}_s dW^{e,j}_s\right)_{i,j,e}.$$ Morever, for all $t\in [0,T],$
$$L^2(\Omega)-\lim_{m\to\infty}\left(\int_0^t (x^m)^{i,e}_s dW^{e,j}_s\right)_{i,j,e}=\left(\int_0^t x^{i,e}_s dW^{e,j}_s\right)_{i,j,e},$$
thanks to the isometry property of It\^o integral. 
Putting these two facts together , we deduce that
$$
{\rm f.d.d.}-\lim_{m\rightarrow\infty}\lim_{n\rightarrow\infty} \left(\sqrt{n}S^{n,i,j,e}_{\cdot,x^m}\right)_{i,j,e} = \left(\int_0^\cdot x^{i,e}_s dW^{e,j}_s\right)_{i,j,e}.
$$
To conclude that ${\rm f.d.d.}-\lim_{n\rightarrow\infty} \left(\sqrt{n}S_{\cdot,x}^{n,i,j,e}\right)_{i,j,e}=\left(\int_0^.x^{i,e}_sdW^{e,j}_s\right)_{i,j,e}$,
and given the decomposition (\ref{deco}), it remains to show 
that 
\begin{equation}\label{remaintoshow}
\lim_{m\rightarrow\infty}\sup_{n\geq m}\sup_{t\in[0,T]}\mathbb{E}[(R^{i,j,e}_{t,m,n,x})^2]=0,\end{equation}  which we do now.

\medskip

We have, for all $t$,
\begin{eqnarray*}
&&\mathbb{E}[(R^{i,j,e}_{t,m,n,x})^2]\\
&=&n\sum_{l,k=1}^{nt_n }\mathbb{E}\left[(x^{i,e}_\frac{k}{n}-(x^m)^{i,e}_\frac{k}{n})(x^{i,e}_\frac{l}{n}-(x^m)^{i,e}_\frac{l}{n})\Delta_{e,j,k,n}(t)\Delta_{e,j,l,n}(t)\right]\\
&=&n\sum_{k=1}^{\lfloor nt \rfloor }\mathbb{E}\left[(x^{i,e}_\frac{k}{n}-(x^m)^{i,e}_\frac{k}{n})^2\Delta_{e,j,k,n}(t)^2\right]\\ 
&&\mbox{(since $x$ is adapted and the  increments of the}\\
&&\mbox{Brownian motion are independent)}\\
&\leq& n\sum_{k=1}^{\lfloor nT \rfloor}\mathbb{E}\left[\|x^{i,e}-(x^m)^{i,e}\|^2_{\infty}\right]\mathbb{E}\left[\Delta_{e,j,k,n}(t)^2\right]\\
&\leq&\frac{T}{2}\mathbb{E}\left[\|x^{i,e}-(x^m)^{i,e}\|^2_{\infty}\right].
\end{eqnarray*}
Let 
\begin{equation*}
N^{i,e}= {\rm Card}\left\{t\in[0,T],|x^{i,e}_t-x^{i,e}_{t-}|+|x^{i,e}_t-x^{i,e}_{t+}|>0\right\},
\end{equation*}
which is almost surely finite because $x$ is piecewise continuous. Let $T_l^{i,e}$ be the $l$-th (random) discontinuity of $x^{i,e}$ ($T_l^{i,e}(\omega)=+\infty$ if $x^{i,e}(\omega)$ has less than $l$ discontinuities over $[0,T]$). It is clear that $T^{i,e}_l$ is measurable as a stopping time. Let $E^m_{i,e}=\cup_{l\in\mathbb{N^*}}(T^{i,e}_l-\frac{1}{m},T^{i,e}_l+\frac{1}{m})\cap[0,T]$. Then, 
\begin{equation*} 
\lim_{m\rightarrow\infty}\sup_{n>m}\frac{{\rm Card}\{k,\frac{k}{n}\in E^m_{i,e}\}}{n}= \lim_{m\rightarrow\infty}(\frac{2N^{i,e}}{m}\wedge 1)=0\quad\mbox{a.s.}.
\end{equation*} 
Observe that $x^{i,e}$ is a.s.\!\! uniformly continuous on $[0,T]\setminus\cup_{l}\{T_l^{i,e}\}$. Moreover, if $s\in(E^m_{i,e})^c$ for some $m$, then there is no discontinuities between $s_m=\frac{\lfloor ms\rfloor}{m}$ and $s$. Then, $$|x_s^{i,e}-(x^m_s)^{i,e}|\leq X^{i,e}_m\mathbb{I}_{(E^m_{i,e})^c}(s)+2\|x^{i,e}\|_{\infty}\mathbb{I}_{E^m_{i,e}}(s),$$ with $$X^{i,e}_m=\sup_{s\in(E^m_{i,e})^c}|x^{i,e}_s-x^{i,e}_{s_m}|.$$ 
Note that $X^{i,e}_m$ is a sequence of square integrables random variables, which converges a.s.\! to $0$ as $m\rightarrow\infty$ and is bounded by the square integrable random variable $2\|x^{i,e}\|_{\infty}$. Finally, we can write
\begin{equation*}
\mathbb{E}[(R^{i,j,e}_{t,m,n,x})^2]\leq\frac{T}{2}\mathbb{E}\left[(X^{i,e}_m)^2+(\frac{4N^{i,e}}{m}\wedge 1)\|x^{i,e}\|_{\infty}^2\right].
\end{equation*} 
The sequence $\left((X^{i,e}_m)^2+(\frac{2N^{i,e}}{m}\wedge 1)\|x^{i,e}\|_{\infty}^2\right)_{m}$ converges to $0$ as $m\rightarrow\infty$, and is bounded by a square integrable random variable. The conclusion
(\ref{remaintoshow}) 
then follows by dominated convergence.

\bigskip

\textit{\underline{Step 3: Tightness}.} Let $0<\mu< \gamma$. We have, for all $i,j,e$, all $s\leq t$ and all $n\in\mathbb{N}^*$,
\begin{eqnarray*}\mathbb{E}\left[\left|S_{t,x}^{n,i,j,e}-S_{s,x}^{n,i,j,e}\right|^{2+\mu}\right]&\leq& |t-s|^{\frac{\mu}{2}}\int_{s}^{t}\mathbb{E}\left[\left|x^{n,i,j}_{s_n}(B^{j,e}_s-B^{j,e}_{s_n})\right|^{2+\mu}\right]ds\\
&\leq& K|t-s|^{1+\frac{\mu}{2}}\mathbb{E}[\|x^{i,j}\|^{2+\gamma}]^\frac{2+\gamma}{2+\mu}
\end{eqnarray*}
where the first inequality is obtained by applying the Burkholder and Jensen inequalities, and the second inequality is obtained by applying the H\"older inequality. This prove the tightness  in $\mathcal{C}_{\mathbb{R}}([0,T])$ of each component of $S^n_x$, and conclude the proof of (\ref{convconv3}).

\end{proof}
\begin{remark} In the case $H=\frac{1}{2}$, notice that the hypothesis $\mathbb{E}[\|x\|_{\infty}^{2+\gamma}]<\infty$ is only needed to obtain the tightness of the process. 
For the convergence of the f.d.d., the hypothesis $\mathbb{E}[\|x\|_{\infty}^2]<\infty$ is sufficient.
\end{remark}

Finally, the following lemma is used in the proof of Proposition \ref{integralProcesses}.
\begin{lemma}\label{driftConvergence}\textit{ 
Let $b$ be a piecewise continuous process such that $\mathbb{E}[\|b\|_{\infty}^{2+\gamma}]<\infty$ for some $\gamma>0$. Then: 
\begin{itemize}
\item For $H>\frac{3}{4}$, in probability uniformly on $[0,T]$,
\begin{equation*}
\nu_H(n)n^{2H-1}\sum_{k=0}^{\lfloor n\cdot\rfloor }b_\frac{k}{n}\int_\frac{k}{n}^{\frac{k+1}{n}\wedge \cdot}(s-s_n)dB_s^{i}\overset{}\longrightarrow\frac{1}{2}\int_0^\cdot b_sdB_s^{i}.
\end{equation*}
\item For $\frac{1}{2}\leq H\leq \frac{3}{4}$, in probability uniformly on $[0,T]$,
\begin{equation*}
\nu_H(n)n^{2H-1}\sum_{k=0}^{\lfloor n\cdot\rfloor }b_\frac{k}{n}\int_\frac{k}{n}^{\frac{k+1}{n}\wedge \cdot}(s-s_n)dB_s^{i}\overset{}\longrightarrow 0.
\end{equation*}
\end{itemize}
}
\end{lemma}

\begin{proof}
The proof in the case $b=\mathbb{I}_{[0,t]}$ is done in \cite{hu2016rate}. Similar arguments as in Lemma \ref{quadraticVariation} allow to conclude.
\end{proof}

\bigskip
\subsection{Proof of Theorem \ref{main1} and \ref{main2}}

\begin{proof}[\nopunct] \textit{Proof of Theorem \ref{main1}:}
For $s\in [0,T]$, recall that $ s_n=\frac{\lfloor ns\rfloor}{n}$ and
\begin{equation*} 
M_t^{n,i,j}=n^{2H-1}\left(\int_0^tu^i_{s}dB^{j}_s-\sum\limits_{k=0}^{\lfloor nt\rfloor}u^i_{\frac{k}{n}}\left(B^{j}_{\frac{k+1}{n}\wedge t}-B^{j}_\frac{k}{n}\right)\right).
\end{equation*}
We have 
\begin{eqnarray*}
M_t^{n,i,j}&=&n^{2H-1}\int_0^t(u^i_{s}-u^i_{s_n})dB^{j}_s\\
&=&n^{2H-1}\int_0^t\sum_{e=1}^dP^{i,e}_{s_n}(B^e_s-B^e_{s_n})dB^{j}_s\\
&&+n^{2H-1}\int_0^t\left(u^i_{s}-u^i_{s_n}-\sum_{e=1}^dP^{i,e}_{s_n}(B^e_s-B^e_{s_n})\right)dB^{j}_s\\
&=&A_t^{n,i,j}+\sum_{e\neq j}R_{t}^{n,e}+R_t^{n,i,j},
\end{eqnarray*} 
with
\begin{eqnarray*}
A_t^{n,i,j}&=&\frac{1}{2}n^{2H-1}\sum_{k=0}^{ nt_n}P^{i,j}_\frac{k}{n}\left(B^{j}_{\frac{k+1}{n}\wedge t}-B^{j}_\frac{k}{n}\right)^2\\
R_{t}^{n,e}&=& n^{2H-1}\sum_{k=0}^{nt_n }P^{i,e}_\frac{k}{n}\int_\frac{k}{n}^{\frac{k+1}{n}\wedge t}\left(B^{e}_s-B^{e}_\frac{k}{n}\right)dB^{j}_s,\quad e\neq j\\
R_t^{n,i,j}&=&n^{2H-1}\int_0^t\left(u^i_s-u^i_{s_n}-\sum_{e=1}^dP^{i,e}_{s_n}(B^e_s-B^e_{s_n})\right)dB^{j}_s.
\end{eqnarray*}

Lemma \ref{quadraticVariation} implies the $L^2(\Omega)$-convergence of $A_t^{n,i,j}$ to $\frac12\int_0^t P_s^{i,j}ds$.
\
We show that all the additional terms converge to $0$ in $L^2(\Omega)$-norm as $n\rightarrow\infty$.
If $e\neq j,$ $D^{1,j}B^{e}=0$, so according to Proposition \ref{traceClass}
\begin{equation*}
\int_\frac{k}{n}^{\frac{k+1}{n}\wedge t}\left(B^{e}_s-B^{e}_\frac{k+1}{n}\right)dB^{j}_s=\int_\frac{k}{n}^{\frac{k+1}{n}\wedge t}\left(B^{e}_s-B^{e}_\frac{k+1}{n}\right)\delta B^{j}_s.
\end{equation*}
Lemma \ref{quadraticVariation} then implies the $L^2(\Omega)$ and UCP convergence of every $R_{t}^{n,e}$ to $0$ for all $e\neq j$ and $t\in [0,T]$.
Moreover, $(u,P)\in\mathbb{C}_1$, so the equation (\ref{pseudo}) implies that
\begin{eqnarray*}
\mathbb{E}\left[\left(R_t^{n,i,j}\right)^2\right]&=&n^{4H-2}\mathbb{E}\left[\left(\sum_{k=0}^{nt_n}L^{i,j}_{\frac{k}{n},\frac{k+1}{n}\wedge t}\right)^2\right]\\
&&=n^{4H-2}\sum_{j=0}^{nt_n}\sum_{k=0}^{nt_n}\mathbb{E}\left[L^{i,j}_{\frac{k}{n},\frac{k+1}{n}\wedge t}L^{i,j}_{\frac{l}{n},\frac{l+1}{n}\wedge t}\right]\\
&=&\epsilon(n)\sum_{j=0}^{nt_n}\sum_{k=0}^{nt_n}r_H\left(\frac{k}{n},\frac{k+1}{n}\wedge t,\frac{j}{n},\frac{j+1}{n}\wedge t\right)\leq T^{2H}\epsilon(n),
\end{eqnarray*}
with $\epsilon(n)\underset{n\rightarrow\infty}\longrightarrow 0$.

Thanks to Lemma \ref{equiContinuity}, we can now show that, for all $i,j\in\{1,\ldots, d\}$, the sequences $(R_t^{n,i,j})_n$  converges UCP to $0$ as $n\rightarrow\infty$, by checking their tightness in $\mathcal{C}_{\mathbb{R}}([0,T])$. We have
\begin{equation}\label{1}R^{n,i,j}_{t}=n^{2H-1}\sum_{k=0}^{nt_n-1}L^{i,j}_{\frac{k}{n},\frac{k+1}{n}}+n^{2H-1}L^{i,j}_{t_n,t}\end{equation}
Thanks to (\ref{pseudo}), we have
\begin{eqnarray*}
\mathbb{E}\left[\left(n^{2H-1}\sum_{k=ns_n}^{nt_n-1}L^{i,j}_{\frac{k}{n},\frac{k+1}{n}}\right)^2\right]&\leq& K\sum_{j,k=ns_n}^{nt_n-1}r_H\left(\frac{k}{n},\frac{k+1}{n},\frac{j}{n}\frac{j+1}{n}\right)\\
&=& K(t_n-s_n)^{2H},
\end{eqnarray*} 
for some $K>0$. Moreover, let $\epsilon\in(0,\alpha-(1-H))$ be small enough (let us recall that $\alpha$ (resp $\beta$) is the H\"older exponent of $u$ (resp $P$)). The second term in the right-hand side of \eqref{1} verifies (due to the regularity and integrability assumptions on $u$ and $P$, as well as the Young-Loeve inequality):
\begin{eqnarray*}
&&\sup_{t\in[0,T]}\left(n^{2H-1}\big|L^{i,j}_{t_n,t}\big|\right)\\
&\leq& c_{\alpha-\frac{\epsilon}{2},H-\frac{\epsilon}{2}}n^{2H-1}n^{-(H+\alpha-\epsilon)}\|B\|_{H-\frac{\epsilon}{2}}\|u^i\|_{\alpha-\frac{\epsilon}{2}}\\
&&+c_{H-\frac{\epsilon}{2},H-\frac{\epsilon}{2}}n^{-1+\epsilon}\|P^{i,j}\|_{\infty}\|B\|^2_{H-\frac{\epsilon}{2}}\underset{n\rightarrow\infty}\longrightarrow 0 \quad\mbox{a.s.}.
\end{eqnarray*}

Then, the sequence $R^{n,i,j}_{t}$ verifies the assumptions of Lemma \ref{tightnessCriterion}, with $A^{n,i,j}_t=n^{2H-1}\sum_{k=0}^{nt_n-1}L^{i,j}_{\frac{k}{n},\frac{k+1}{n}}$, $C^{n,i,j}_t=n^{2H-1}L^{i,j}_{t_n,t}$, $\alpha_0=2H-1,\beta_0=2$, which proves the tightness.
\end{proof}

\begin{proof}[\nopunct] \textit{Proof of Theorem \ref{main2}:}
1. Let $H>\frac{1}{2}$. Again, we can write 
\begin{equation*}
M^{n,i,j}_t-\frac{1}{2}\int_0^tP_s^{i,j}ds=M_{j,t}^{n,i,j}+\sum_{e\neq j}M_{e,t}^{n,i,j}+R_{1,t}^{n,i,j}+R_{2,t}^{n,i,j},
\end{equation*}
where, for $1\leq i\leq m$ and $1\leq j\neq e\leq d$,
\begin{eqnarray*}
M_{j,t}^{n,i,j}&=&n^{2H-1}\sum_{k=0}^{nt_n}P^{i,j}_\frac{k}{n}\left(\int_\frac{k}{n}^{\frac{k+1}{n}\wedge t}\left(B^{j}_s-B^{j}_\frac{k}{n}\right)dB^{j}_s-\frac{((k+1)\wedge t-k)}{2n^{2H}}\right)\\ 
M_{e,t}^{n,i,j}&=&n^{2H-1}\sum_{k=0}^{nt_n}P^{i,e}_\frac{k}{n}\int_\frac{k}{n}^{\frac{k+1}{n}\wedge t}\left(B^{e}_s-B^{e}_\frac{k}{n}\right)dB^{j}_s\\ 
R_{1,t}^{n,i,j}&=&\frac{1}{2}\left(\frac{1}{n}\sum_{k=0}^{nt_n-1}P^{i,j}_\frac{k}{n}+P^{i,j}_{nt_n}(t-t_n)-\int_0^tP^{i,j}_sds\right)\\
R_{2,t}^{n,i,j}&=&n^{2H-1}\int_0^t\left(u^i_s-u^i_{s_n}-\sum_{e=1}^dP^{i,e}_{s_n}(B^e_s-B^e_{s_n})\right)dB^{j}_s.
\end{eqnarray*}
Since
$(u,P)\in\mathbb{C}_2$, we have that $\nu_H(n)R_{2,t}^{n,i,j}\overset{L^2(\Omega)}{\underset{n\rightarrow\infty}\longrightarrow}0$
by using again the formula \eqref{pseudo}. The tightness of the sequence $(\nu_H(n)R_{2,t}^{n,i,j})_n$ can be proved by using the same argument as in the previous proof.

\medskip On the other hand, since $P$ is $\beta$-H\"older continuous for some $\beta>\frac{1}{2}$ we have that $\sup_{t\in [0,T]}\left|\nu_H(n)R^{n,i,j}_{1,t}\right|\rightarrow 0$ a.s., which guarantees the convergence of $\nu_H(n)R^{n,i,j}_{1,\cdot}$ to $0$ in $\mathcal C_{\mathbb{R}}([0,T])$.
When $H>\frac{3}{4}$, since we have the additional hypothesis that $\sum_{i,j}\mathbb{E}[\|P^{i,j}\|_{\beta}^{2+\gamma}]<\infty$ for some $\gamma>0$, we can further prove the $L^2(\Omega)$ convergence: for each $t\leq T$, $\nu_H(n)R^{n,i,j}_{1,t}\overset{L^2(\Omega)}{\underset{n\rightarrow\infty}\longrightarrow}0$.

\medskip Finally, using Proposition \ref{traceClass} we observe that 
\begin{equation*}\int_\frac{k}{n}^{\frac{k+1}{n}\wedge t}\left(B^{j}_s-B^{j}_\frac{k}{n}\right)dB^{j}_s-\frac{(nt-nt_n)}{2n^{2H}}=\int_\frac{k}{n}^{\frac{k+1}{n}\wedge t}\left(B^{j}_s-B^{j}_\frac{k}{n}\right)\delta B^{j}_s\notag\end{equation*}
\smallbreak\noindent and, if $e\neq j$, \begin{equation*}\int_\frac{k}{n}^{\frac{k+1}{n}\wedge t}\left(B^{e}_s-B^{e}_\frac{k}{n}\right)dB^{j}_s=\int_\frac{k}{n}^{\frac{k+1}{n}\wedge t}\left(B^{e}_s-B^{e}_\frac{k}{n}\right)\delta B^{j}_s.\end{equation*}
Since $P$ verifies the regularity assumptions of Lemma \ref{skorokhodVariation}, we get the stated convergence for all values of $H>\frac{1}{2}$:
\begin{itemize}\item  If $\frac{1}{2}<H\leq\frac{3}{4}$,
$$\left\{\nu_H(n)\left(M_{\cdot}^{n,i,j}-\frac{1}{2}\int_0^\cdot P^{i,j}_sds\right)\right\}_{i,j}\underset{n\rightarrow\infty}{\longrightarrow}\left\{\int_0^\cdot P_s^{i,e}dW^{e,j}_s\right\}_{e,i,j}$$
where the convergence holds in $\mathcal{C}_{\mathbb{R}^{d\times m}}([0,t])$.
\item If $H>\frac{3}{4}$,
$$\left\{\nu_H(n)\left(M_{\cdot}^{n,i,j}-\frac{1}{2}\int_0^\cdot P^{i,j}_sds\right)\right\}_{i,j}\underset{n\rightarrow\infty}{\longrightarrow}\left\{\int_0^\cdot P_s^{i,e}dZ_s^{e,j}\right\}_{e,i,j}$$
where the convergence holds UCP (and in $L^2(\Omega)$ for fixed $t$).
\end{itemize}

2. Once the necessary modifications are made, the proof is the same for Brownian motion.
\end{proof}

\noindent{\bf Acknowledgements}. I. Nourdin and V. Garino are supported by the FNR OPEN grant APoGEe at Luxembourg University.
We heartily  thank the editorial board and the two reviewers, whose comments and careful reading led to a drastic change from the original version and significantly improved the readability of the article.

\bibliographystyle{plain}

\end{document}